\documentclass[12pt]{amsart}
\usepackage{amscd,amssymb,latexsym}
\usepackage{fullpage}
\newcommand{\Mdef}[2]{\newcommand{#1}{\relax \ifmmode #2 \else $#2$\fi}}

\newcommand{\height}{\mathrm{ht}}
\newcommand{\im}{\mathrm{im}}

\newcommand{\sm }{\wedge}

\newcommand{\tensor}{\otimes}

\newcommand{\sdr}{\rtimes}

\newcommand{\Hom}{\mathrm{Hom}}

\Mdef{\bhom}{\mathbf{\hat{H}om}}

\Mdef{\Mod}{\mathrm{mod}}

\newcommand{\st}{\; | \;}

\newtheorem{thm}{Theorem}[section]
\newtheorem{lemma}[thm]{Lemma}
\newtheorem{prop}[thm]{Proposition}
\newtheorem{cor}[thm]{Corollary}

\theoremstyle{definition}

\newtheorem{defn}[thm]{Definition}

\newtheorem{warning}[thm]{Warning}

\newtheorem{example}[thm]{Example}

\newtheorem{remark}[thm]{Remark}

\newcommand{\qqed}{\qed \\[1ex]}
\renewenvironment{proof}[1][\hspace*{-.8ex}]{\noindent {\bf Proof #1:\;}}{\qqed}

\Mdef{\PH} {\Phi^H}
\Mdef{\PK} {\Phi^K}
\Mdef{\PL} {\Phi^L}
\Mdef{\PT} {\Phi^{\T}}

\Mdef{\ef}{E{\cF}_+}
\Mdef{\etf}{\widetilde{E}{\cF}}
\Mdef{\eg}{E{G}_+}
\Mdef{\etg}{\tilde{E}{G}}

\Mdef{\infl}{\mathrm{inf}}
\Mdef{\defl}{\mathrm{def}}
\Mdef{\res}{\mathrm{res}}
\Mdef{\ind}{\mathrm{ind}}
\Mdef{\coind}{\mathrm{coind}}

\Mdef{\univ}{\mathcal{U}}

\Mdef{\Fp}{\mathbb{F}_p}
\Mdef{\Zpinfty}{\Z /p^{\infty}}
\Mdef{\Zpadic}{\Z_p^{\wedge}}

\newcommand{\adjunction}[4]{
\diagram
#1:#2 \rrto<0.7ex> &&
#3  \llto<0.7ex> :#4 
\enddiagram}
\newcommand{\recollement}[3]{
\diagram 
#1
 \rrto<0.0ex> 
 &&
 #2 \llto<0.7ex>
 \llto<-0.7ex>
 \rrto<0.0ex>
  &&
 #3
\llto<0.7ex>
 \llto<-0.7ex>
\enddiagram}

\newcommand{\lra}{\longrightarrow}
\newcommand{\lla}{\longleftarrow}

\newcommand{\lr}[1]{\langle #1 \rangle}

\newcommand{\Gspectra}{\mbox{$G$-{\bf spectra}}}

\Mdef{\we}{\mathbf{we}}
\Mdef{\fib}{\mathbf{fib}}
\Mdef{\cof}{\mathbf{cof}}
\Mdef{\BI}{\mathcal{BI}}

\newcommand{\ilim}{\mathop{ \mathop{\mathrm{lim}} \limits_\leftarrow} \nolimits}
\newcommand{\colim}{\mathop{  \mathop{\mathrm {lim}} \limits_\rightarrow} \nolimits}
\newcommand{\holim}{\mathop{ \mathop{\mathrm {holim}} \limits_\leftarrow} \nolimits}

\newcommand{\wilim}{\mathop{ \mathop{\mathrm{lim}^w} \limits_\leftarrow} \nolimits}

\newcommand{\A}{\mathbb{A}}
\Mdef{\B}{\mathbb{B}}
\Mdef{\C}{\mathbb{C}}
\Mdef{\D}{\mathbb{D}}
\Mdef{\E}{\mathbb{E}}
\Mdef{\T}{\mathbb{T}}
\Mdef{\F}{\mathbb{F}}
\Mdef{\G}{\mathbb{G}}
\Mdef{\I}{\mathbb{I}}
\Mdef{\N}{\mathbb{N}}
\Mdef{\Q}{\mathbb{Q}}
\Mdef{\R}{\mathbb{R}}
\Mdef{\bbS}{\mathbb{S}}
\Mdef{\Z}{\mathbb{Z}}

\Mdef{\bA}{\mathbb{A}}
\Mdef{\bB}{\mathbb{B}}
\Mdef{\bC}{\mathbb{C}}
\Mdef{\bD}{\mathbb{D}}
\Mdef{\bE}{\mathbb{E}}
\Mdef{\bF}{\mathbb{F}}
\Mdef{\bG}{\mathbb{G}}
\Mdef{\bH}{\mathbb{H}}
\Mdef{\bI}{\mathbb{I}}
\Mdef{\bJ}{\mathbb{J}}
\Mdef{\bK}{\mathbb{K}}
\Mdef{\bL}{\mathbb{L}}
\Mdef{\bM}{\mathbb{M}}
\Mdef{\bN}{\mathbb{N}}
\Mdef{\bO}{\mathbb{O}}
\Mdef{\bP}{\mathbb{P}}
\Mdef{\bQ}{\mathbb{Q}}
\Mdef{\bR}{\mathbb{R}}
\Mdef{\bS}{\mathbb{S}}
\Mdef{\bT}{\mathbb{T}}
\Mdef{\bU}{\mathbb{U}}
\Mdef{\bV}{\mathbb{V}}
\Mdef{\bW}{\mathbb{W}}
\Mdef{\bX}{\mathbb{X}}
\Mdef{\bY}{\mathbb{Y}}
\Mdef{\bZ}{\mathbb{Z}}

\Mdef{\cA}{\mathcal{A}}
\Mdef{\cB}{\mathcal{B}}
\Mdef{\cC}{\mathcal{C}}
\Mdef{\mcD}{\mathcal{D}} 
\Mdef{\cE}{\mathcal{E}}
\Mdef{\cF}{\mathcal{F}}
\Mdef{\cG}{\mathcal{G}}
\Mdef{\mcH}{\mathcal{H}} 
\Mdef{\cI}{\mathcal{I}}
\Mdef{\cJ}{\mathcal{J}}
\Mdef{\cK}{\mathcal{K}}
\Mdef{\mcL}{\mathcal{L}}

\Mdef{\cM}{\mathcal{M}}
\Mdef{\cN}{\mathcal{N}}
\Mdef{\cO}{\mathcal{O}}
\Mdef{\cP}{\mathcal{P}}
\Mdef{\cQ}{\mathcal{Q}}
\Mdef{\mcR}{\mathcal{R}}
\Mdef{\cS}{\mathcal{S}}
\Mdef{\cT}{\mathcal{T}}
\Mdef{\cU}{\mathcal{U}}
\Mdef{\cV}{\mathcal{V}}
\Mdef{\cW}{\mathcal{W}}
\Mdef{\cX}{\mathcal{X}}
\Mdef{\cY}{\mathcal{Y}}
\Mdef{\cZ}{\mathcal{Z}}

\Mdef{\ca}{\mathcal{a}}
\Mdef{\ct}{\mathcal{t}}

\Mdef{\At}{\tilde{A}}
\Mdef{\Bt}{\tilde{B}}
\Mdef{\Ct}{\tilde{C}}
\Mdef{\Et}{\tilde{E}}
\Mdef{\Ht}{\tilde{H}}
\Mdef{\Kt}{\tilde{K}}
\Mdef{\Lt}{\tilde{L}}
\Mdef{\Mt}{\tilde{M}}
\Mdef{\Nt}{\tilde{N}}
\Mdef{\Pt}{\tilde{P}}

\Mdef{\tA}{\tilde{A}}
\Mdef{\tB}{\tilde{B}}
\Mdef{\tC}{\tilde{C}}
\Mdef{\tE}{\tilde{E}}
\Mdef{\tH}{\tilde{H}}
\Mdef{\tK}{\tilde{K}}
\Mdef{\tL}{\tilde{L}}
\Mdef{\tM}{\tilde{M}}
\Mdef{\tN}{\tilde{N}}
\Mdef{\tP}{\tilde{P}}

\Mdef{\ft}{\tilde{f}}
\Mdef{\xt}{\tilde{x}}
\Mdef{\yt}{\tilde{y}}

\Mdef{\Ab}{\overline{A}}
\Mdef{\Bb}{\overline{B}}
\Mdef{\Cb}{\overline{C}}
\Mdef{\Db}{\overline{D}}
\Mdef{\Eb}{\overline{E}}
\Mdef{\Fb}{\overline{F}}
\Mdef{\Gb}{\overline{G}}
\Mdef{\Hb}{\overline{H}}
\Mdef{\Ib}{\overline{I}}
\Mdef{\Jb}{\overline{J}}
\Mdef{\Kb}{\overline{K}}
\Mdef{\Lb}{\overline{L}}
\Mdef{\Mb}{\overline{M}}
\Mdef{\Nb}{\overline{N}}
\Mdef{\Ob}{\overline{O}}
\Mdef{\Pb}{\overline{P}}
\Mdef{\Qb}{\overline{Q}}
\Mdef{\Rb}{\overline{R}}
\Mdef{\Sb}{\overline{S}}
\Mdef{\Tb}{\overline{T}}
\Mdef{\Ub}{\overline{U}}
\Mdef{\Vb}{\overline{V}}
\Mdef{\Wb}{\overline{W}}
\Mdef{\Xb}{\overline{X}}
\Mdef{\Yb}{\overline{Y}}
\Mdef{\Zb}{\overline{Z}}

\Mdef{\db}{\overline{d}}
\Mdef{\hb}{\overline{h}}
\Mdef{\qb}{\overline{q}}
\Mdef{\rb}{\overline{r}}
\Mdef{\tb}{\overline{t}}
\Mdef{\ub}{\overline{u}}
\Mdef{\vb}{\overline{v}}

\Mdef{\hc}{\hat{c}}
\Mdef{\he}{\hat{e}}
\Mdef{\hf}{\hat{f}}
\Mdef{\hA}{\hat{A}}
\Mdef{\hH}{\hat{H}}
\Mdef{\hJ}{\hat{J}}
\Mdef{\hM}{\hat{M}}
\Mdef{\hP}{\hat{P}}
\Mdef{\hQ}{\hat{Q}}

\Mdef{\thetab}{\overline{\theta}}
\Mdef{\phib}{\overline{\phi}}

\Mdef{\uA}{\underline{A}}
\Mdef{\uB}{\underline{B}}
\Mdef{\uC}{\underline{C}}
\Mdef{\uD}{\underline{D}}

\Mdef{\bolda}{\mathbf{a}}
\Mdef{\boldb}{\mathbf{b}}
\Mdef{\bfD}{\mathbf{D}}

\Mdef{\fm}{\frak{m}}
\Mdef{\fp}{\frak{p}}

\newcommand{\fD}{\mathfrak{D}}
\newcommand{\fX}{\mathfrak{X}}
\newcommand{\fY}{\mathfrak{Y}}
\newcommand{\fZ}{\mathfrak{Z}}

\Mdef{\eps}{\epsilon}

\newcommand{\cell}{\mathrm{Cell}}

\input{xypic}
\usepackage{graphicx}
\newcommand{\sub}{\mathrm{Sub}}
\newcommand{\PP}{\mathbb{P}}

\newcommand{\cIi}{\cI^{-1}}

\newcommand{\Zt}{\tilde{\Z}}

\newcommand{\End}{\mathrm{End}}
\renewcommand{\tb}{\overline{\times}}

\newcommand{\full}{\mathrm{full}}

\newcommand{\diag}{\mathrm{diag}}

\newcommand{\Rt}{\tilde{R}}
    \newcommand{\cospan}{\lrcorner}

\newcommand{\modules}{\mbox{-mod}}

\newcommand{\Rcospan}{R^{\cospan}}

\newcommand{\piG}{\pi^G}

\newcommand{\shv}{\mathbf{Shv}}
\newcommand{\cWshv}{\cW\!-\!\mathbf{Shv}}
\newcommand{\Ubar}{\overline{U}}

\newcommand{\etale}{\'etale}

\newcommand{\Rtop}{R^{top}}
\newcommand{\Rttop}{\Rt^{top}}
\newcommand{\RfX}{R_{\fX}}
\newcommand{\RcWfX}{R_{\fX}[\cW]}
\newcommand{\cAhat}{\hat{\mathcal{A}}}
\newcommand{\cAhyb}{\widetilde{\mathcal{A}}}
\newcommand{\cAsep}{\mathcal{A}_{sep}}

\newcommand{\Nhat}{\hat{N}}
\renewcommand{\height}{\mathrm{ht}}
\newcommand{\Qu}{\underline{\Q}}
\newcommand{\cFt}{\widetilde{\cF}}

\begin{document}
\title{Rational $G$-spectra over blocks with finite Weyl groups}
\author{J.P.C.Greenlees}
\address{Mathematics Institute, Zeeman Building, Coventry CV4, 7AL, UK}
\email{john.greenlees@warwick.ac.uk}

\date{}

\begin{abstract}
We show that for any clopen collection $\fX$ of subgroups of $G$ with
finite Weyl groups, the category of  $G$-spectra with geometric
isotropy in $\fX$ is equivalent to the category of equivariant sheaves
over $\fX$. This gives an algebraic model of injective dimension at
most the rank of $G$.
\end{abstract}

\thanks{The author is grateful to D.Barnes and R.Biswas for conversations related
  to this paper. The work is partially supported by EPSRC Grant
  EP/W036320/1. The author  would also  like to thank the Isaac Newton
  Institute for Mathematical Sciences, Cambridge, for support and
  hospitality during the programme Equivariant homotopy theory in
  context, where most  of  work on this paper was undertaken. This work was supported by EPSRC grant EP/Z000580/1.  }
\maketitle

\tableofcontents

\section{Overview}

\subsection{Introduction}
If we pick a set $\fX$ of conjugacy classes of (closed) subgroups of  a compact 
Lie group $G$ we can consider the category $\Gspectra\lr{\fX}$ of  $G$-spectra whose 
geometric isotropy is contained in $\fX$, and the present paper gives
a complete and calculable model for a broad class of examples where
the structure is especially simple. 

The space $\sub(G)/G$ of conjugacy classes of closed subgroups
of $G$ has two topologies of interest.
First of all, it has the h-topology, which is the quotient topology of the
Hausdorff metric topology on $\sub (G)$. This is a compact, Hausdorff,
totally disconnected space, shown in \cite{prismatic} to be of
Cantor-Bendixson rank at most equal to the rank of $G$. The space also
has the Zariski topology, whose closed sets are the h-closed sets which are also
closed under passage to cotoral subgroups\footnote{$L$ is {\em cotoral} in $K$
if $L$ is normal in $K$ with quotient being a torus. For conjugacy
classes $(L)$ is cotoral in $(K)$ if the relation holds for some
representative subgroups.}.

If we decompose the space $\sub(G)/G$ as a union of Zariski clopen
subsets
$$\sub(G)/G=\fX_1\amalg \cdots \amalg \fX_n,$$
then idempotents in the Burnside ring show that the category of
rational $G$-spectra decomposes into blocks
$$\Gspectra \simeq \Gspectra \lr{\fX_1}\times  \cdots \times \Gspectra
\lr{\fX_n}$$
where  $\Gspectra \lr{\fX}$ is the category of $G$-spectra with
geometric isotropy in $\fX$.

It is conjectured 
\cite{AGconj} that 
there is a small and calculable abelian model $\cA (G|\fX)$ so that the category DG-$\cA 
(G|\fX)$ of differential graded objects in $\cA (G|\fX)$ is a model 
for $\Gspectra\lr{\fX}$ in the sense that there is a Quillen equivalence 
$$\Gspectra\lr{\fX}\simeq DG-\cA (G|\fX).$$

In this paper we show the conjecture is correct for  blocks
$\Gspectra \lr{\fX}$  in the simplest possible case: the space
$\fX$ is open and closed in the Zariski topology with all subgroups $H\in \fX$ having finite Weyl
group.

All blocks of finite groups are all of this type, and
we can take $\fX=\sub(G)/G$. The simplest non-trivial
example is when $G=O(2)$ and $\fX=\{ (D_{2n})\st n\geq 1\}\cup \{
O(2)\}$ consists of the dihedral subgroups and the limit point
$O(2)$. Other 1-dimensional examples occur in \cite{gq1}, and we will
describe some more in Section
\ref{sec:examples}. 

\subsection{The algebraic model}
The algebraic model incorporates the monodromy action of the finite
Weyl groups $\cW_{(K)}=W_G(K)$ for $K\in \fX$. The first thing to note
is that this involves continuous selection of a representative $K$ for the
conjugacy class $(K)$, respecting containment; it will be proved in
\cite{gqblocks} that it is always possible to do this.
In Part 2 we will describe the additional data
precisely, codified as a `component structure'. There are in fact at
least four
different descriptions of the algebraic model.

\begin{thm}
  \label{thm:threemodels}
Suppose $\fX$ is a clopen subset $\fX$ of subgroups of $G$ with finite Weyl 
group. The  following four abelian categories are equivalent
\begin{itemize}
\item the category $\cW-\shv/\fX$ of $\cW$-equivariant sheaves of 
$\Q$-modules. 
\item the complete (or separated) model $\cAhat (\fX, \cW)$ given by stalks and 
  splicing data (sometimes called `the 
  $V$-model' after the names of stalks; see Definition \ref{defn:Vmodel}). 
\item the category of modules over the ring $\Gamma (\fX ;\Q[\cW])$
  of global sections of the sheaf of rings $\Q[\cW]$ with stalk $W_G(K)$
  over $(K)$. 
\item the standard abelian category $\cA (\fX, \cW)$ (sometimes called
  `the $N$-model' after the names of the modules, see Definition
  \ref{defn:Nmodel}).
\end{itemize}
 The injective dimension of the four abelian categories is equal to the 
Cantor-Bendixson rank of $\fX$. This dimension depends only on $\fX$; it is finite 
and at most the rank of the group $G$. 
\end{thm}

This is a special case of Theorem \ref{thm:cWshvscatStone}, which
shows the corresponding fact for any Stone space with finite
Cantor-Bendixson rank. 

The main result shows that any of the four abelian categories of Theorem
\ref{thm:threemodels} provide an algebraic model of rational $G$-spectra
over $\fX$.

\begin{thm}
  Suppose $\fX$ is a clopen subset $\fX$ of the space of conjugacy
  classes of subgroups of $G$ with finite Weyl 
  group. 
  
   The category $\Gspectra\lr{\fX}$ of rational $G$-spectra with 
  geometric isotropy in $\fX$ is Quillen equivalent to DG objects in
  any one of the four equivalent abelian categories of Theorem
  \ref{thm:threemodels}. These are of injective dimension at most the
  rank of $G$.
\end{thm}

We view the equivalence
$$\Gspectra\lr{\fX}\simeq \cW-\shv/\fX$$
as the main result of the paper. The $V$-model $\cAhat (\fX, \cW)$ is
really another version of sheaves, and  the other two abelian models are
mentioned because the equivalence is proved using the
standard model $\cA (\fX, \cW)$, and it may be considered the module
description is more elementary.

The strategy is the obvious one. It is shown in \cite{gfreeq2,
  spcgq} that the model is correct on each stalk, and the local
structure of $\fX_G=\sub(G)/G$ is described in \cite{t2wqalg,
  gqblocks}. The machinery of adelic assembly can then be used to show
the equivalence is global.

\subsection{Strategy}
We divide the work into three parts.

In Part 1, we study the category of sheaves over 
$\fX$. Of course $\fX$ is a very special sort of space: it is a Stone 
space (compact, Hausdorff and totally disconnected) and of finite 
Cantor-Bendixson rank (By \cite{prismatic} this applies when $\fX=\fX_G$ is 
the space of conjugacy classes of subgroups of a compact Lie group 
$G$). It is not surprising that sheaves over such a 
space have a simple algebraic model, but the standard model appears to 
offer a new point of view and connects with the adelic philosophy. 

The brief Part 2 explains how to add Weyl equivariance.  We
introduce the notion of a component  structure, and observe
that the work of Part 1 is compatible with the 
additional group actions. Philosophically, this is unsurprising but
there is content in the reduction of  a  $G$-equivariant sheaf
over a $G$-space such as $\sub (G)$ to a sheaf with finite Weyl group
actions over the quotient space is of
some interest, and the identification of small generators needs attention.

In Part 3, we show that the algebraic arguments of Parts 1 and 2  lift to the 
category of $G$-spectra, and that the ring $G$-spectra that occur are 
formal in a strong sense. First, the existence of good representatives 
of conjugacy classes from \cite{gqblocks} gives us the appropriate 
component structure. Calculations of homotopy together with 
Shipley's Theorem show that there is a model in terms of diagrams of 
simpler categories. However moving to algebra whilst retaining the
component structures is the trickiest bit of the process. Smallness of
generators is delicate, and and we need the use of bimodule Morita
equivalences to obtain the appropriate diagram. Taking all these
ingredients in parallel in the algebraic and topological cases, we
establish an equivalence with the specific  category $\cA (\fX, \cW)$.

 \subsection{Relationship to other work}
The present paper depends on understanding of the structure of the
space $\sub(G)$ of subgroups of $G$, as described in \cite{t2wqalg}
and \cite{gqblocks}. This is enough to show the local structure in
algebra and topology agrees, and general methods of assembly apply in
parallel in the two cases. 

We briefly explain how the present article fits into the general
programme to give an algebraic model for rational $G$-spectra. As 
described above, it is natural to work block by block. 

The first work on a Weyl-finite block is the case $G=O(2)$ in \cite{o2q} (at the
triangulated level) and \cite{BarnesO(2)} (at the level of model
structures). The case of 1-dimensional blocks with a single element at
height 1 is treated in \cite{gq1}. Isolated subgroups occur for finite
groups,  but next for $G=SO(3)$, which is treated in \cite{so3q} (at the
triangulated level) and \cite{KedziorekSO(3)} (at the level of model
structures). 

The complement to the behaviour of Weyl-finite blocks is that of
Noetherian blocks.  This is illustrated by work on tori
\cite{tnqcore}, and the general case will be treated in
\cite{AGnoeth}. 

To see how these two complentary behaviours occur, we may consider
$O(2)$ with one Weyl-finite block and one Noetherian block. One may 
consider the group $SU(3)$ \cite{t2wqalg, gq1, t2wqmixed, u2q, su3q}. 
This has 18 blocks: 6 are Noetherian,  9 are isolated, 2 are
Weyl-finite and 1 is mixed. 

Work is in progress to describe the mixed case in general.

\subsection{Conventions}
We will work throughout in the category of rational spectra, so all
homotopy groups and maps of spaces are rational vector spaces.

We will constantly consider diagrams indexed on the poset $C$ of
subsets of $[r]=\{0,1,2,\ldots, r\}$, which is an $(r+1)$-cube, and the
punctured cube $PC$ consisting of nonempty subsets of $[r]$. We
will be careful to distinguish between $C$-diagrams and $PC$-diagrams
without overburdening the notation. We are interested
in the limit of $PC$-diagrams: if $X$ is a $PC$-diagram we may extend
to a $C$-diagram by defining $X_{\emptyset}=\ilim_{A\in PC}X_A$,  
and we say the resulting $C$-diagram
$X$ is  a limit diagram.

\subsection{Contents}
Part 1 is focused on the abelian category of sheaves over a scattered Stone space.
Section \ref{sec:ringcubes} introduces the cube of localized product
rings.  Section \ref{sec:shvcubes} introduces the corresponding diagram in
sheaves; the constructions may seem more natural, and most of our
proofs take place in the world of sheaves. In Section
\ref{sec:shvadelic} we make explicit the adelic model in the context
of abelian groups and sheaves. In Section \ref{sec:stalkstandard} we make
explicit the standard model (the `$N$-model) and the complete model
(also called the complete model, the stalk model or the $V$-model)
and explain how the two models are equivalent at the level of abelian
categories.  Finally in Section \ref{sec:dim1} we make the constructions and
associated algebra explicit in the 1-dimensional case. It is
sobering to see the delicacy of the situation even in this simple
case.

Part 2 introduces component structures and explains how the
constructions from  Part 1 can be upgraded to permit equivariance.
In Section \ref{sec:componstone} the categories of equivariant objects are defined, and in Section
\ref{sec:gens} continuity conditions are brought under sufficient
control to identify generators. 

Part 3 establishes that the category of $G$-spectra over $\fX$ is
equivalent to the abelian category introduced in Parts 1 and 2.
Section \ref{sec:Rtop} introduces the spectral lifting of
the algebraic rings. Section \ref{sec:topgens} explains how component
structures are brought in to the construction. The key here is to
identify the generators, and their endomorphism ring, and explain how
Morita theory can be incorporated into the pullback cube.
Finally Section \ref{sec:examples} illustrates the generality of our
results, and makes this explicit in a few low dimensional examples. 

\section{Panoramic overview}
This section describes the overall strategy for proving the main
theorem, giving the outline and highlighting the trickier points. 

\subsection{Sheaves, the standard model and the complete model}
Let us first discuss the rare situation that all Weyl groups are 
trivial. This lets us show the landscape without distraction. 

We are working in three categories: $G$-spectra with geometric 
isotropy in $\fX$, modules over $\Q$ and sheaves of $\Q$-modules over $\fX$. 
If we write $\Qu$ for the constant sheaf, routine calculations give 
$$[S^0,S^0]^G_*=\Gamma (\fX , \Qu)=H^*(\fX, \Qu).$$
We emphasize that all three rings are entirely in degree 0. This connects the three categories. 

In fact if $R=\Gamma (\fX , \Qu)=C(\fX ,\Q)$ we 
are comparing $S^0$-modules (in $G$-spectra over $\fX$), $R$-modules 
(in $\Q$-vector spaces) and $\Qu$-modules (in sheaves over $\fX$). The
simplifying feature of the 
trivial Weyl group situation, is that each of these three categories of modules 
is generated by the ring itself. This is automatically small in the
first two cases, and the constant sheaf is small because $\fX$ is
compact. 

Since the three categories of modules are generated by the ring which 
has the same endomorphism ring $R$ in degree 0, Morita theory gives equivalences 
$$\Gspectra\lr{\fX}=S^0\modules/\fX \simeq R\modules \simeq \Qu\modules=\shv/\fX.$$
Even the first of these equivalences gives an algebraic model for 
$G$-spectra. However this does not complete the task, since to make calculations we 
still need to assemble the categories from simpler pieces. One method
of assembly is adelic, and one is the reconstruction of a sheaf from
its stalks.

For the adelic apprach we proceed as follows.
If $\A$ is any one of the three rings ($S^0, R$ or $\Qu$), 
we may construct it as the homotopy pullback of a punctured cube  of 
rings as follows.  We suppose $\fX$ is of Cantor-Bendixson rank $r$, we may 
consider the  $(r+1)$-cube of subsets $A\subseteq [0,r]$, and we form 
the cubical diagram $\A_{\bullet}$ with $\A_{\emptyset}=\A$. We will show by very 
standard techniques that it is a 
homotopy pullback in all three cases. There is a cubical diagram of 
module categories, and this leads to an algebraic model. In one
language, the model consists of a weighted limit of the punctured
cubical diagram, so that objects consist of a $PC$-diagram of modules,
cocartesian in the sense that each module is taken to the next. In
this introduction we view such a diagram of modules as a
section of the diagram of module categories, so we have  a Quillen equivalence 
$$\A\modules \simeq \Gamma (PC, \A_{\bullet}\modules)$$
of model categories, where $PC$ denotes the punctured cube indexed by 
non-empty subsets $A$. 

By another standard homotopy calculation $\pi^G_*(S^0_A)=R_A$, entirely 
in degree 0. Since we are assuming trivial Weyl groups, $S^0_A$ is
also a generator, so  we have a Quillen equivalence
$$S^0_A\modules \simeq R_A\modules$$
for all $A\subseteq [0,r]$. It follows that the section models for 
$S^0$-modules and $R$-modules are equivalent. 
The diagram of module categories in
algebra arises from a diagram of abelian categories, and one may use
this to show that the abelian categories give a
cellular skeleton: the section model consists of differential
graded objects in the {\em standard model} (Definition \ref{defn:Nmodel})
$$\cA (\fX)= \Gamma_{ab} (PC, R_{\bullet} \modules). $$
The subscript $ab$ indicates we are taking the {\em abelian} category 
of diagrams $N_A$ which are cocartesian in the sense that
$N_B=R_B\tensor_{R_A}N_A$ whenever $A\subseteq  B$. 
This is the first abelian model. 

However when it comes to the diagrams of rings and of sheaves, 
$$R_A\modules \not \simeq \Qu_A\modules$$
unless $A=\emptyset$ or the space is compact in some sense. To start
with, the generator $\Qu_A$ is usually not small. 
For example, if $A=\{ s\}$ is a singleton we will see that
$R_s=\prod_y \Q$,  the product ring (over an indexing set
$\mcD_s$ that is usually countably infinite). On the other hand,  
$\Qu_s$ is the skyscraper sheaf over $\mcD_s$ (as a discrete space). Thus $R_s$-modules 
consists of modules over the product ring, whilst the category 
$\Qu_s$-modules is the product of the module categories $\Q$-mod:
these are equivalent only if $\mcD_s$ is finite.

This means that the section model coming from sheaves is
different. Nonetheless, it too has a cellular skeleton,  the {\em
  hybrid model}
$$\cAhyb (\fX)= \Gamma_{ab} (PC, \Qu_{\bullet}\modules).  $$
This seems to be of limited usefulness, and we focus instead on the
{\em complete model}  $\cAhat (\fX)$ (Definition \ref{defn:Vmodel}); this uses the well known encoding 
of a sheaf as an \'etale
space from its stalks together with germ-spreading data.

\subsection{Generators}
The situation with Weyl groups $W_G(K)$ non-trivial (but still finite)
is a little more complicated. We may once again
find single small generators $g_A$ for $S^0_A$-modules provided $A\neq
\emptyset$, so that $[g_A,
g_A]^G_*$ is entirely in degree 0, and it is again the global sections of a
sheaf $\Qu_A[\cW]$ over $\fX$. As the notation suggests, the sheaf $\Qu
[\cW]$ has stalk $\Q W_G(K)$ over $K$, and similarly for other subsets
$A$. We write $R_A[\cW]=\Gamma (\fX; \Qu_A[\cW]))$, omitting mention
of $A $ when it is empty.

The three categories under consideration are  $G$-spectra over $\fX$
(now with finite Weyl groups unrestricted), $R[\cW]$-modules, 
and $\cW$-equivariant sheaves over $\fX$ (i.e., $\Qu[\cW]$-modules).
When $A\neq \emptyset$, each 
 category still has a single small generator. For spectra and rings the
 generator is small, but for sheaves it is generally not small.  
For spectra and rings, the endomorphism rings
are the same and entirely in degree 0
$$[g_A,g_A]^G_*=\Gamma (\fX , \Qu_A[\cW])=H^*(\fX, \Qu_A[\cW]).$$
Morita theory again gives the equivalence 
$$S^0_A\modules/\fX \simeq R_A[\cW]\modules .$$ 
This time there is no convenient generator for $A=\emptyset$ in the
category of spectra, and we focus on the adelic rings $S^0_A$ to obtain a model
for $S^0$-modules over $\fX$ to obtain a model. The simple form of the
argument requires a careful choice of generators $g_A$, and they do
not get mapped to each other under extensions of scalars. We therefore
need to use a bimodule Morita equivalence to show that we obtain a comparison map of
punctured cubical diagrams. Once that is done, 
we may once again use adelic approximations, and since they agree
for the sphere and rings we obtain a model in terms of algebra.
Taking the cellular skeleton, this gives the
{\em standard model}
$$\cA (\fX, \cW)= \Gamma_{ab} (PC, R_{\bullet}[\cW]\modules), $$
where the subscript indicates we are taking the {\em abelian}
category. The Weyl groups play a largely independent role so we can
say that the objects are diagrams $N_{\bullet}$  of
$R_{\bullet}[\cW]$-modules which are cocartesian in the sense that
$N_B=R_B\tensor_{R_A}N_A$ whenever $A\subseteq  B$. Since
$R[\cW]$-modules are equivalent to modules over $\Qu[\cW]$
($A=\emptyset$ only!), we also have the adelic model of sheaves, which again
has a cellular skeleton and we obtain the hybrid model 
$$\cAhyb (\fX, \cW)= \Gamma_{ab} (PC,
\Qu_{\bullet}[\cW]\modules). $$
In fact the hybrid model seems to be of limited usefulness, and
instead we use the complete (or separated) model $\cAhat(\fX,\cW)$ which builds
sheaves from stalks.

\subsection{Roadmap 1}
The overall diagram of equivalences is
$$\xymatrix{
  &\Gspectra\lr{\fX}\ar@{=}[r]&S^0\modules/\fX\rto^(0.45){\simeq}&  \Gamma(PC, S^0_{\bullet}\modules)\ar@{<->}[d]^{\simeq}&\\
&  &R[\cW]\modules/\fX\rto^(0.45){\simeq}&
  \Gamma(PC, R_{\bullet}\modules)\ar@{<-}[r]^(0.6){\simeq}&\cA(\fX, \cW)\\
\cAhat (\fX, \cW) \rto^(0.45){\simeq}&\cW-\shv/\fX \ar@{=}[r] &\Qu [\cW]\modules/\fX\rto^(0.45){\simeq}\uto_{\simeq}&
  \Gamma(PC, \Qu_{\bullet}\modules)\ar@{<-}[r]^(0.6){\simeq} \uto &\cAhyb(\fX, \cW)\uto
}$$

Two omissions in the diagram are worth comment.

First, there is no direct
equivalence between $S^0$-modules and $R[\cW]$-modules, since this
would require us to find a suitable generator in spectra with $R[\cW]$ as its
endomorphisms.

Second, the two upward maps on the right indicaed by $*$ are not
equivalences at the abelian level. 

Finally, it is tempting to give a direct route between  the standard model $\cA
(\fX, \cW)$ and  the complete model $\cAsep (\fX, \cW)$. We make this
explicit in dimension 1, but in higher dimensions it is left implicit
as a composite of several other maps. 

\part{Sheaves over scattered Stone spaces}

\section{The cube of rings}
\label{sec:ringcubes}
We suppose that $\fX$ is a compact, Hausdorff, totally disconnected 
topological space with  a finite filtration 
$$\emptyset=\fX_{-1}\subset \fX_0\subset \fX_1\subset \cdots \subset \fX_r=\fX$$
with each $\fX_i$ open and dense in $\fX$, and so that the pure strata 
$$\fD_i:= \fX_i\setminus \fX_{i-1}$$
are discrete. It follows that the subsets of $\fD_i$ that are closed
in $\fX$ are the finite subsets. 

The fundamental structure arising from this is  a punctured
$(r+1)$-cube of rings.

\begin{example}
In the motivating example, $\fX$ is  an open and closed subset of the space of
conjugacy classes of closed subgroups of a compact Lie group $G$. The
entire space of conjugacy classes of closed subgroups is
compact Hausdorff and totally disconnected. It is
shown in \cite{prismatic} that the space has the given structure where
$r$ is the rank of the group.
\end{example}

\begin{thm}
  \label{thm:shvscatStone}
Suppose $\fX$ is a Stone space of finite Cantor-Bendixson rank. 
The  following four abelian categories are equivalent 
\begin{itemize}
\item the category $\shv/\fX$ of  sheaves of 
$\Q$-modules. 
\item the complete model $\cAhat (\fX, \cW)$ given by stalks and 
  splicing data  (sometimes called `the 
  $V$-model' after the names of stalks; see Definition \ref{defn:Vmodel}). 
\item the category of modules over the ring $\Gamma (\fX )$
  of global sections of the constant sheaf 
\item the standard abelian category $\cA (\fX)$ (sometimes called `the
  $N$-model after the name of the modules; see 
  Definition  \ref{defn:Nmodel}).
\end{itemize}
 The injective dimension of the three abelian categories is equal to the 
Cantor-Bendixson rank of $\fX$. 
\end{thm}

\subsection{The splicing rings}
\label{subsec:splicingrings}
First, if we have an indexing set $U$ and rings $R_y$ for $y\in U$, 
we may consider a product  $\prod_{y\in U}R_y$. If $V\subseteq U$ we
write $e_U\in \prod_yR_y$ for the idempotent with support $U$. We now
suppose $U$ is a subspace of a larger space $\fX$. Given a point $x$  (for example a
limit point of $U$),  we may consider the multiplicatively closed set
$$\cI_{x|U}=\{ e_V\st V=U\cap N, N \mbox{ and open neighbourhood of
  $x$ } \} \subseteq \prod_y R_y. $$

We suppose given a nonempty subset  $A=\{ a_0, a_1, \ldots , a_s\}$ of the set
$[r]=\{0, 1, \cdots , r\}$ of heights, arranged in decreasing order
$a_0>a_1>\cdots > a_s$. To this we can associate the ring
$$R_A=\prod_{\height(H_0)=a_0} \cIi_{H_0|a_1}
\prod_{\height(H_1)=a_1} \cIi_{H_1|a_2} \cdots 
\prod_{\height(H_{s-1})=a_{s-1}}\cIi_{H_{s-1}|a_s}\prod_{\height(H_s)=a_s} \Q. $$

Thus an element of $R_A$ is represented by an equivalence class of functions $f: \fD_{a_0}\times \cdots
\times \fD_{a_s} \lra \Q$, where the equivalence relation is described
by the neighbourhoods of points in 
$$\fX_{a_0}\supseteq
\fX_{a_1}\supseteq \cdots \supseteq \fX_{a_s}.$$

In Section \ref{sec:shvcubes} we will give an interpretation in
terms of sheaves.

\subsection{The adelic diagram of rings}
  We now observe that the rings $R_A$ for $\emptyset\neq A\subseteq [r]$ assemble
  into a punctured $(r+1)$-cube. We need only observe that if
  $B\subseteq [r]$ is obtained by adding
  one height $b$ to $A$ there is a map $d^b: R_A\lra R_B$.

  \begin{example}
    If $r=1$ we obtain the diagram 
    $$\xymatrix{
      &R_1\dto \\
      R_0\rto & R_{10}.
      }$$

          If $r=2$ we obtain the diagram 
    $$\xymatrix{
      &R_1\ddto\rrto &&R_{21}\ddto  \\
      && R_{2}\urto \ddto \\
      &R_{10}\rrto &&R_{210}\\
      R_0\urto\rrto &&R_{20}\urto.
      }$$
    \end{example}

    To describe the maps $d^b$, we consider first the case that
    $b$ is less than all the elements of $A$ so that
    $B=\{ a_0>a_1> \cdots >a_s>b\}$ in that case,
    for each $H_s\in \mcD_s$ we use the localized diagonal map 
  $$M\lra \prod_{\height(H_b)=b}M\lra 
  \cIi_{H_s|b}\prod_{\height(H_b)=b}M $$
  with $M=\Q$.

Because the   construction of the rings $R_A$ is iterative, this case
gives all cases. If $B=\{ a_0>a_1>  \cdots >a_t>b>a_{t+1}>\cdots >a_s
\}$ we take
$A'=\{ a_0>\cdots >a_t\}$ and $A''=\{ a_{t+1}>\cdots >a_s\}$. We start
with the identity map  $R_{A''}\lra R_{A''}$ then apply the
above construction for 
$R_{A''}\lra R_{bA''}$ and then apply the same iterations on both
factors to give
$$d^b: R_A=R_{A'A''}\lra R_{A'bA''}=R_B. $$

Finally, we may complete the cube by defining
$$R_{\emptyset}=\Gamma (\fX ; \Qu_{\fX})$$
to be the global sections of the constant sheaf. This maps to the
rings $R_a$ by evaluation of a section at the stalks.

\subsection{Vanishing of adelic cohomology}
The punctured cube gives the adelic chain complex \cite{adelicc}
$$0\lra C^0\lra C^1\lra \cdots \lra C^{r+1}\lra C^{r+2}=0$$
where (recalling that an $i$-simplex has $i+1$ vertices) we take 
$$C^i=\prod_{|A|=i+1} R_A$$
and where the differential $d: C^i\lra C^{i+1}$ is the sum of  components
$(-1)^{A|b}d^b: R_A\lra R_B$ where $A|b$ is the position of $b$ in $A$
(i.e. 0 if $b>a_0$, 1 if $a_0>b>a_1$ and so forth). The punctured cube
can be augmented to a cube by adding the ring of global sections
$\Gamma (\fX; \Q)$  at the initial point, and this corresponds to the augmented complex
$$0\lra \Gamma (\fX, \Q)\lra C^0\lra C^1\lra \cdots \lra C^{r+1}\lra
0. $$
The main  input to our argument is the following. 

\begin{lemma}
  \label{lem:ringcubeexact}
The augmented complex is exact.
\end{lemma}

In principle one could give a direct proof of this using cochains. The result
is immediate in rank 1, but it is harder to keep track of detailed
arguments in higher rank. We will give a more digestible argument
using sheaves in Section \ref{sec:shvcubes} (Corollary \ref{cor:hopullcube}).

\section{The cube of sheaves}
\label{sec:shvcubes}

We show that the cube of rings from Section \ref{sec:ringcubes} can be
lifted to sheaves, and thereby prove the exactness statement.

\subsection{Restriction and extension}
We write $\Qu_U$ for the constant sheaf on $U$. If $U$ is discrete,
its ring of sections is the product: $\Gamma (\Qu_U)=\prod_{u\in
  U}\Q$. We may consider an element of this product as a
function $f: U\lra \Q$.

We will need to consider $U$ in a larger space $\Ubar$ in
which $U$ is dense and open, and we write $i: U\lra \Ubar$ for the
inclusion. We have the restriction functor $i^*: \shv/\Ubar\lra
\shv/U$. This is also the restriction in  terms of \etale\ spaces, 
in the sense that $i^*\cF$ has the same stalks as $\cF$ on $U$.
This functor has a right adjoint $i_*: \shv/U \lra \shv/\Ubar$. The
sheaf $i_*\cF$ has the same stalks over $U$ as $\cF$ does, but over
$x\in \Ubar$ it has the stalk
$$(i_*\cF)_x=\colim_{A\ni x}\cF(A\cap U). $$

In particular if $U$ is discrete and  $\Qu_U$ is the constant sheaf over $U$ we have
$$(i_*\Qu_U )_x=\colim_{A\ni x}\prod_{A\cap U}\Q,   $$
where the limit is over open sets $A$.
It is convenient to express this as a localization, so we take the
ring $\Gamma \Qu_U=\prod_{u\in U}\Q$ and the multiplicatively closed
set
$$\cI_{x|U}=\{ e_{V\cap U}\st V \mbox{ is an clopen neighbourhood of }
x\}.$$
Thus
$$(i_*\Qu_U )_x=\colim_{A\ni x}\prod_{A\cap U}\Q=\cIi_{x|U}\prod_{u\in
  U}\Q,  $$
just as in Subsection \ref{subsec:splicingrings}. 
In other words, an element is an equivalence class of functions $f:
U\lra \Q$ where $f\sim f'$ if the two functions agree in a
neighbourhood of $x$.

\subsection{Lifting the rings to sheaves}
\label{subsec:lifttoshv}

We recall that the space $\fX_s$ is open in $\fX$, and we name
the inclusions
$$\xymatrix{
\fX_s\rto^{i_s}&\fX&\lto_{j_s}\fX\setminus \fX_s.
}$$

\begin{defn}
Given  $A=\{a_0>a_1>\cdots 
  >a_s\} \subseteq [r]$, we may  construct a sheaf $\mcR_A$ on $\fX$
  as follows. First, if $A=\emptyset$ we take
  $\mcR_{\emptyset}=\Qu_{\fX}$. Now, if $s=0$ we  start with  the sheaf 
  $\Qu_{a_0}$ on $\fD_{a_0}$ and extend it to $\fX_{a_0}$; this 
  does not change any stalks since $\fD_{a_0}$ is closed in 
  $\fX_{a_0}$, so we continue to write $\Qu_{a_0}$.

    Now if 
  $d_0A=\{a_1>\cdots >a_s\}$ and $\mcR_{d_0A}$ is defined 
  we take 
  $$\mcR_A=i^{a_0}_*i_{a_0}^*\mcR_{d_0A}. $$
  Since $i^{a_0}_*$ is a right adjoint, there is a unit map 
  $$\mcR_{d_0A}\lra i^{a_0}_*i_{a_0}^*\mcR_{d_0A} =\mcR_A.  $$

Displaying all the maps,   we have 
  $$\xymatrix{
                                      &\shv/\fX\drto^{i^*_{a_{s-1}}}&
                                      &\cdots&\cdots\drto^{i^*_{a_0}}&&\shv/\fX\\
   \shv/\fX_{a_s}\urto^{i^{a_s}_*} &      &\shv/\fX_{a_{s-1}}\urto^{i^{a_{s-1}}_*}&&&\shv/\fX_{a_0}\urto^{i^{a_0}_*}&
}$$
and 
  $$\mcR_A=i^{a_0}_*i_{a_0}^*\cdots i_{a_{s-1}}^*i^{a_s}_*\Qu_{a_s}. $$
\end{defn}

This construction is related to that of Subsection \ref{subsec:splicingrings}
by taking global sections:
\begin{lemma}
\label{lem:GammaRAisRA}
  The sheaf and ring constructions are related by taking global sections
   $$\Gamma (\mcR_A)=R_A, $$
   and this also applies to the maps in the cube. 
 \end{lemma}

 \begin{proof} If $A=\emptyset$ this is true by definition. Otherwise
   we proceed as follows.  It is clear if $A=a_s$ by construction. 
  We then repeatedly use two facts.

   Since $i^a_*$ is right adjoint to restriction we have 
   $$\Gamma (\fX, i^a_*\cF)=\Gamma(\fX_a, \cF)$$
   for sheaves $\cF$ on $\fX_a$. 

   On the other hand if $\cG$ is supported on $\fD_a$ we have 
   $$\Gamma(\fX_a, \cG)=\prod_{x_a\in \fD_a}\cG_{x_a}. $$

   Thus 
   $$\Gamma(\fX, \mcR_A)=\Gamma(\fX, i^{a_0}_*j^*_{a_0}\mcR_{d_0A})=
   \Gamma(\fX_{a_0}, 
   j^*_{a_0}\mcR_{d_0A})=\prod_{x_0}\cIi_{x_0|a_1}\prod_{x_1}
   (\mcR_{d_0A})_{x_1}$$
      \end{proof}

   \begin{lemma}
     \label{lem:soft}
All sheaves of $\Q$-modules on $\fX$ are soft and hence acyclic.
 \end{lemma}

 \begin{proof}
Softness is proved in the paper \cite{Wiegand} (where `Stone space' is
used in a more permissive sense). Indeed any closed set is an
intersection of a decreasing sequence of clopen sets.  If $U$ is
clopen we have $\cF(\fX)=\cF(U)\oplus \cF(U^c)$, so sections over
$\cF$ always extend. If $K=\bigcap_i U_i$ then
$$\cF (K)=\colim_i \cF(U_i)$$
is a colimit of a sequence of epimorphisms and hence $\cF(\fX)\lra
\cF(K)$ is an epimorphism. 

By \cite[II.9.11]{BredonSheaf}, since the family of all closed subsets of 
$\fX$ is paracompactifying,  soft sheaves are acyclic. 
\end{proof}

 \begin{remark}
(i) The constant sheaf is usually not flabby (the set $\fX_0$ is discrete
and open, so sections over it are uncountable, whereas sections over
$\fX$ are countable).

(ii) By contrast to the acyclicity of all sheaves, the injective dimension is equal to the Cantor-Bendixson rank
\cite{BarnesSugrue}. 
 \end{remark}  

 \begin{lemma}
   \label{lem:acycliccube}
     The cube of sheaves $\mcR_A$ is acyclic, or equivalently
     $$\Qu_{\fX}\simeq \holim_{A\neq \emptyset}\mcR_A.$$
   \end{lemma}

   \begin{proof}
Since all the terms are acyclic by  Lemma
\ref{lem:soft}, we need only verify the cube of sheaves is exact, and
this can be done stalkwise. If $x\in \fD_a$ then
$(\mcR_A)_x=0$ if $a_0>a$, so we only need to consider $\mcR_A$ for
$A\subseteq [a]$. Furthermore, if $a\not \in A$ the map
$$(\mcR_A)_x\stackrel{\cong}\lra (\mcR_{Aa})_x. $$
This means that the entire limit is from $\emptyset \lra A=\{a\}$, where we evidently
have the isomorphism
$$\Q =(\Qu_{\fX})_x\lra (R_{\{a\}})_x=(\Qu_a)_x=\Q.$$
     \end{proof}
   
     \begin{cor}
       \label{cor:hopullcube}
            The homotopy limit of the punctured cube $R_{\bullet}$ of
            rings is its limit
            $$\Gamma (\Qu_{\fX})\simeq \holim_{A\neq \emptyset}
            R_A,  $$
            the ring of global sections of the constant sheaf. 
\end{cor}

\begin{proof}
Taking global sections, we have a spectral sequence for  calculating
the cohomology of the homotopy pullback sheaf. Since the sheaves
$\Qu_{\fX}$ and $\mcR_A$ are soft, the spectral sequence
collapses to the complex of global sections. Since $H^0(\fX,
\mcR_A)=R_A$ this is the augmented adelic complex. Since the 
constant sheaf is soft, the complex is exact. 
  \end{proof}

  \begin{warning}
    \label{warn:RARA}
The categories of modules over the sheaf $\mcR_A$ and the $\Q$-algebra 
$R_A$ will generally not be equivalent. For example $\mcR_0$ is the constant sheaf over the
discrete space $\fD_0$ so its module category is the product
$\prod_{d\in \fD_0}(\Q\modules)$. On the other hand the category of
modules over $R_0=\prod_{d\in
  \fD_0}\Q$ is much larger (for example it includes the module $\prod
\Q/\bigoplus\Q$, which has trivial stalks). 
    \end{warning}

       \section{Adelic description of sheaves}
\label{sec:shvadelic}
For scattered Stone spaces $\fX$, one may show directly that the
abelian category of sheaves over $\fX$ is  equivalent to the category 
of modules over the ring of global sections $\Gamma (\Qu_{\fX})$. This
just amounts to saying the reconstruction of a sheaf from an \etale\
space only involves constructions that can be expressed using idempotents in the ring. 

The present section gives an alternative and less direct approach that 
parallels the construction of the algebraic model of $G$-spectra over
$\fX$ in Part 3. 
In Sections \ref{sec:ringcubes} and \ref{sec:shvcubes} we defined
punctured cubes of rings, one in 
the category of vector spaces and one in the category of sheaves over 
$\fX$. The first gives a model for modules over $\Gamma (\Qu_{\fX})$ in 
terms of modules over diagrams of simpler rings. The second
gives a model for sheaves over $\fX$ in terms of modules over diagrams
of sheaves of rings. These two diagram categories are different but equivalent.

\subsection{The category $\cA (\fX)$}
In Section \ref{sec:ringcubes} we introduced the punctured cubical diagram $\{ R_{A}\}_{A\neq 
  \emptyset}$ of rings, and Corollary \ref{cor:hopullcube} shows that we have
$$\Gamma (\Qu_{\fX})=\ilim_A R_A\simeq \holim_A R_A.$$

We now form a related  abelian category,  $\cA(\fX)$ that will be used
in the comparison. We follow  \cite{s1q} et seq for the terminology.
\begin{defn}
The {\em standard model} on $\fX$ is the
category $\cA (\fX)$ of cocartesian $R_{\bullet}$-modules.

In more detail, an $R_{\bullet}$-module is a punctured cubical diagram $M_{\bullet}$
where $M_A$ is an $R_A$-module and if $A\subseteq B$ the map
$M_A\lra M_B$ is a map of modules over $R_A\lra R_B$. The morphisms 
$M_{\bullet}\lra N_{\bullet}$ are given by maps of diagrams. A
{\em cocartesian} module is one in which the maps $M_A\lra M_B$ induce an
isomorphism $R_B\tensor_{R_A} M_A\cong M_B$.

\end{defn}

\begin{lemma}
  \label{lem:cAfXcoref}
The inclusion $i: \cA (\fX)\lra R_{\bullet}\mbox{-mod}$ is coreflective:
it has a right adjoint $\Gamma$ so that the unit $T\lra \Gamma i T$ is
an isomorphism. 
\end{lemma}

\begin{proof}
Since localization and tensor products commute with colimits, $i$
preserves colimits and a right adjoint exists. 
  \end{proof}

\subsection{Diagrams of module categories}
The diagram $R_\bullet$ also gives rise to a diagram
$R_{\bullet}$-modules of module categories. The maps relating the
categories are extensions of scalars, and as such they have left
adjoints. We may therefore apply the theory of \cite{diagrams} to
give the category of $R_{\bullet}$-modules a model structure. As in
\cite[4.1]{diagrams} the cellularization of  the diagram of module
categories gives a model for modules over the homotopy inverse limit
ring. As in \cite[9.5]{adelicm} this is the category of homotopy
cocartesian diagrams (indicated here by the superscript $w$ for
`weighted').

  \begin{prop}
    The category of $R_{\bullet}$-modules admits the algebraically and 
    diagrammatically injective model structure, and we may cellularize 
    it with respect to the diagram $R_{\bullet}$. There are Quillen 
    equivalences 
    $$\Gamma (\Qu_{\fX})\mbox{-mod}\simeq 
    \cell_R R_{\bullet}\mbox{-mod} \simeq \ilim^w_AR_A\mbox{-mod} \simeq DG-\cA(\fX).  $$
    \end{prop}

    \begin{proof}
First of all, we saw in Corollary \ref{cor:hopullcube} that the ring $\Gamma
(\Qu_{\fX})$ is the homotopy limit of the diagram. By the
Cellularization Principle \cite[4.1]{diagrams}
this gives the first equivalence. The second is \cite[9.5]{adelicm},
and the last is the Cellular Skeleton Theorem this follows as in
\cite{adelic1}:  by Lemma \ref{lem:cAfXcoref} $DG-\cA(\fX)$ is a
coreflective subcategory, and the counit is a
an equivalence on the injective cogenerators: indeed, the cogenerators
correspond to the sheaves constant on pure strata, and the objects in
$\cA (\fX)$ corresponding to the constant sheaves on $\fD_i$ are those `constant' 
 on $A\subseteq [0,i]$. 
\end{proof}

\subsection{Counterparts for sheaves}
We have described results based on a diagram of {\em rings} (in
abelian groups),  but in view of Lemma \ref{lem:soft}, 
they apply equally well to the diagram $\{ \mcR_{A}\}_{A\neq 
  \emptyset}$ of rings in the category of sheaves over $\fX$ to give
the hybrid category $\cAhyb(\fX)$.

  \begin{prop}
    The category of $\mcR_{\bullet}$-modules admits the algebraically and 
    diagrammatically injective model structure, and we may cellularize 
    it with respect to the diagram $\mcR_{\bullet}$. There are Quillen 
    equivalences 
    $$\shv/\fX = \Qu_{\fX}\mbox{-mod}\simeq 
    \cell_R \mcR_{\bullet}\mbox{-mod}\simeq \ilim^w_A\mcR_A\mbox{-mod} \simeq DG-\cAhyb(\fX).  $$
    \end{prop}

    \subsection{Equivalence of the standard models}
    One might guess that an equivalence would be established by 
     showing the corresponding module categories over $R_A$ and over 
     $\mcR_A$ are equivalent, but in fact this is generally not true
     (Warning \ref{warn:RARA}). 
    
     \begin{lemma}
       If $\fY\subseteq \fX$ then passage to global sections
       $$\Gamma (\fY; \cdot ): \shv/\fY\lra \Gamma(\Qu_{\fY})\modules$$
       is the inclusion of a reflective subcategory. The
       reconstruction $e$ of an \etale\ space from its module of global sections is
       left adjoint to $\Gamma (\fY; \cdot)$ and the counit is an
       isomorphism
       $\cG\cong e\Gamma (\fY;\cG)$.

       If $\fY$ is compact, the unit is also an equivalence. 
     \end{lemma}

     \begin{proof}
The functor $e$ is defined on stalks by $(eM)_x=\colim_{U\ni
  x}e_UM$. Since $\fX$ has a basis of open and closed sets, 
 is easy to see that the counit is an isomorphism $(e\Gamma \cF)_x=\colim_{U\ni x}
 \cF(U)=\cF_x$. 

Consider  the unit $\eta: M\lra \Gamma e M$. If $\eta (a)=0$ then for
every $x\in \fX$ there is a neighbourhood $U_x$ with
$e_{U_x}(a)=0$. Taking a finite subcover we express $\fX$ as a
disjoint union of finitely many pieces on which $a$ is zero, so that
$a=0$. For surjectivity, we suppose given a section $\sigma$ and for
each $x\in \fX$ we pick a clopen neighbourhood $U_x$ of $x$  with
$\sigma_x$ represented as $e_{U_x}a_x$. Taking a finite subcover we
find $\fX$ as a disjoint union of clopen sets on which $\sigma $ is
represented.

       \end{proof}

       Since $\fX$ is compact, we may immediately deduce that two abelian categories of
       concern are equivalent. 
  \begin{cor}
    There is a Quillen equivalence of DG-categories
    $$DG-\shv/\fX \simeq DG-\Gamma (\Qu_{\fX})\modules.$$
\end{cor}

\section{The complete model and the standard model}
\label{sec:stalkstandard}
We now have numerous models based on the punctured cube, and the
author has found it all too easy to jump to unjustified or incorrect
conclusions about their relationship to each other, so we give a
systematic description here.

\subsection{Roadmap 2}
Recall that
$\mcR=\Qu$ is the constant sheaf over $\fX$, and $R=\Gamma
(\mcR)=\Gamma(\Qu)$ is the ring of global sections.  
$$\xymatrix{
V\mbox{-model}= \cAhat(\fX)
\rto^(0.65){\simeq}&\shv/\fX \ar@{=}[r]&\Qu\modules \rto^{\simeq}_{5.5}\dto^{\simeq}_{4.5}&
 R\modules\dto^{\simeq}_{4.6}\\
 &&\ilim^w_A \mcR_A \rto^{*\simeq*} \dto^{\simeq}&\ilim^w_A R_A\dto^{\simeq} \\
 &&DG- \cAhyb(\fX)\rto^{*\simeq*} &DG-\cA (\fX) \ar@{=}[d]\\
&&& N \mbox{-model}
 }$$


 The most useful two models are the standard model ($N$-model) and the
 complete model ($V$-model). The complete model is most familiar,
 being the description of sheaves as \'etale spaces, and the standard
 model is the one that makes contact with the category of
 $G$-spectra. 

The  main point of displaying the diagram is to point out that
 it takes many steps to get from the $V$-model to the $N$-model, and
 also to emphasize that the two marked equivalences are not
 pointwise equivalences of underlying diagrams.

\subsection{The complete, separated, stalk or $V$-model}
The $V$-model is so called because the foreground data consists of vector 
spaces $V(x)$ for $x\in \fX$. In other words, this is the model of sheaves as \'etale 
spaces. They are specified by stalks $V(x)$ together with 
germ spreading data. This makes it natural to call it the `stalk
model'. We should also make contact with the terminology of
\cite{adelic1}, where  this would be called the `separated model'
because the data at different stalks is initially separate. In the
current Weyl-finite context, the separated model is automatically
complete, and we have adopted `complete model' as the default
terminology since it emphasizes all the good features of the model at
once. 

When it comes to defining the category we do so in the most concrete terms.

\begin{defn}
  \label{defn:Vmodel}
  The {\em complete (or separated) model}  $\cAhat(\fX)$ is the category of vector space
  stalks over   $\fX$, so that for every $x\in \fX$ we have the stalk $V(x)$. 

If $x$ is of height $s$ and $t\leq s$ we have a germ-spreading map 
$$\sigma: V(x)\lra \cIi_{(x|t)}\prod_{\dim (y)=t} V(y).$$
The set $(x|t)$ is a collection of subsets of $\mcD_t$, namely the 
interesections of an open neighbourhood of $x$ with $\mcD_t$. 
From this we build the   multiplicatively closed set $\cI_{(x|t)}$ of 
idempotents corresponding to  subsets in $(x|t)$.

The germ spreading is transitive in the sense that if $s\geq t \geq u$
the square 
$$\xymatrix{
V(x)\rto\dto &\cIi_{(x|t)}\prod_{\dim (y)=t} V(y)\dto \\
\cIi_{(x|u)}\prod_{\dim (z)=u} V(z)\rto &
\cIi_{(x|t)}\prod_{\dim (y)=t} \cIi_{(y|u)}\prod_{\dim (z)=u}V(z) 
}$$
commutes, where the lower horizontal is induced by the diagonal. The 
diagram may be constructed since a neighbourhood of $x$ containing $y$ is a 
neighbourhood of $y$.
\end{defn}

\subsection{The standard or $N$-model}
The $N$-model is so called because the foreground data consists of modules
$N_A$ for non-empty subsets $A\subseteq [0,r]$.
To even specify this we need to start with the punctured cube of rings
$R_A$, so that $N_A$ is an
$R_A$-module.

\begin{defn}
  \label{defn:Nmodel}
The {\em standard model} is a cocartesian diagram of modules. 
For every nonempty $A\subseteq [0,r]$ we have an $R_A$-module 
$N_A$. These form an $R_{\bullet}$-module $N_{\bullet}$ 
in the sense that if $A\subseteq B$ there is a map $\ell_A^B: N_A\lra N_B$ which 
is a module map over $i_A^B: R_A\lra R_B$. It is cocartesian in the sense 
that $\ell_A^B$ induces an isomorphism 
$$R_B\tensor_{R_A}N_A= N_B. $$

This is transitive in the sense that if $A\subseteq B \subseteq C$
then the composite 
$$N_C=R_C\tensor_{R_A}N_A=R_C\tensor_{R_B}R_B \tensor_{R_A} N_A$$
is the given isomorphism. 
\end{defn}

\begin{remark}
We have used the terminology `cocartesian' to make contact with
similar structures elsewhere. In the author's earlier work the
condition is separated into extensions $A\lra B$ where $B$ adds
initial (large) elements, called {\em quasicoherence}, and those where it
adds terminal (small) elements called {\em extendedness}. There is no reason
to separate these here, and the cocartesian condition is the same as
being   quasi-coherent and extended (qce). 
  \end{remark}



\subsection{From $N$ to $V$}
It is rather easy to go from the $N$-model to the $V$-model:
$e:\cA (\fX) \lra \cAhat (\fX) $ taking the $N$-data to the $V$-data via
idempotents. Thus if $x$ is of height $s$
$$(eN)(x)= e_xN_s. $$
If $x$ is of height $s$ and $t\leq s$, the germ spreading map is 
$$\xymatrix{
V(x)=e_xN_s\rto  &R_{st}\tensor_{R_{st}} e_xN_s\dto&&
e_xR_{st}\tensor_{R_t} \prod_y e_y N_t\ar@{=}[r]&e_xR_{st}\tensor_{R_t} \prod_y V(y)\\
&R_{st}\tensor_{R_{st}} N_s\rto^{\cong}&N_{st}& R_{st}\tensor_{R_t} N_t\lto_{\cong}
\uto&
}$$
Since $R_{st}=\prod_x\cIi_{x|t}\prod_y\Q$, 
$$e_xR_{st}\tensor_{R_t} \prod_y V(y)=\cIi_{(x|t)}\prod_y V(y). $$

\subsection{From $V$-model to sections}
The functor $V\mbox{-model} \lra \shv(\fX)$ converting an \'etale
space to a presheaf of sections is well known, but it is worth making
it explicit in our case. 

One may expect the sections over the subset $\fX_{\geq s}$ to be
significant. However, its relationship to $N_s$ is perhaps not as
close as expected.

To define $\cF(\fX_{\geq s})$ we consider the $(r-s+1)$-cube of subsets of $[s,r]$,
and form the diagram of objects  $V_{As}$ where $A=(a_0>a_1>\cdots
>a_n)$ with $a_n\geq s$ and 
$$V_{(A|s)}=\prod_{\dim (x_0)=a_0} \cIi_{x_0|a_1)} \prod_{\dim
  (x_1)=a_1} \cIi_{(x_1|a_2)}
\cdots \prod_{\dim (x_{n-1})=a_{n-1}}\cIi_{x_{n-1}|a_n)}\prod_{\dim (x_n)=a_n} V(x_n)$$
and 
$$\cF(\fX_{\geq s})=\ilim_A N_{(s|A)}. $$

\begin{example} We take $r=2$ and suppose there is a unique element
$x$ of height 2, and if we write $x,y$ for generic elements
of height 1, 0, we have  $\fX(\fX_{\geq 2})=V(x)$, 
$$\xymatrix{\cF(\fX_{\geq 1}\rto \dto & V(x) \dto \\
\prod_{\dim (y)=0} V(y)\rto & \cIi_{(x|0)}\prod_{\dim (y)=0} V(y) 
}$$

$$\xymatrix{
&\prod_y V(y)\rrto \ddto && \cIi_{(x|1)} \prod_y V(y) \ddto \\
\cF(\fX_{\geq 0})\rrto \ddto \urto&& V(x)\urto \ddto & \\
&\prod_y \cIi_{(y|0)} \prod_zV(z)\rrto
&&\cIi_{(x|1)}\prod_y\cIi_{(y|0)} \prod_z V(z)\\
\prod_z V(z)\urto \rrto &&\cIi_{(x|0)}\prod_z V(z)  \urto &
}$$

The germ-spreading maps are the three verticals not involving
$N_0$. The other maps not involving $N_0$ are constructed from
diagonals and extension of scalars. 
\end{example}

\subsection{From $V$ to $N$}
It is much less easy to identify the $N$-model corresponding to a 
sheaf. We first build the sheaf $\cF_V$ from the $V$-model $V$, which we may think 
of as an $\mcR$-module. 
 We take global sections to get an $R$-module $\cF(\fX)$, and extend 
scalars to get an $R_A$-module $R_A\tensor_R\cF(\fX)$. 

The main point to make is that in either case the module $N_A$ does 
not have stalk $\cF_x$ but rather $(\mcR_A)_x\tensor \cF_x$.

The alternative is to  extend scalars to make $\mcR_A$-module 
$\mcR_A\tensor_{\mcR} \cF$, and then take global sections to get 
$$N_A=\Gamma (\mcR_A\tensor_{\mcR}\cF) .$$
It is clear that this is closely related to the other route, and that
it agrees on idempotent summands. However interchange of products and
localizations make a direct comparison delicate.

   \section{Dimension 1}
   \label{sec:dim1}
The 1-dimensional case is very attractive because we can be completely
explicit, but most of the interesting features already appear. In
particular, it is useful to see explicit counterexamples to many naive expectations. 

\subsection{The context}
We suppose the topological space $\fX$ is the one-point compactification of a
countably infinite set $\cF$. Because of the examples we have in mind,
we write $G$ for the compactifying point. This has a filtration with
pure strata $\fD_0= \cF$, and $\fD_1=\{ G\}$.

In this case there are at least five distinct points of view on
the abelian category giving the algebraic model for $G$-spectra with
geometric isotropy in $\fX$. There are three in algebra: the standard
model, the complete model and the category of modules over the ring of
functions on $\fX$. There are two in sheaves: the category of
sheaves of $\Q$-modules and the complete model (stalks plus
germ-spreading). The point about
mentioning all five is that they do embody
different points of view: the functors between them make
significant changes and are only equivalences because of special
algebraic features of this case.

\subsection{Rings in algebra}
In the category of rational vector spaces we consider the 
rings 
$$R_1=\Q^{\fD_1}=\Q \mbox{  and } R_0=\Q^{\fD_0}=\prod_{F\in 
  \cF}\Q. $$
We also consider the multiplicatively closed set 
$$\cI=\{ e: \fD_0\lra \{ 0,1\} \st e =1 \mbox{ almost everywhere }
  \}$$
  of idempotents with cofinite support, 
and  we write $R_{10}=\cIi R_0$. This gives a diagram $R_{\bullet}$ of rings 
$$\xymatrix{
  &\Q\dto \\
  \prod_{\cF} \Q\rto &\cIi \prod_{\cF}\Q 
}$$
We note that there is a short exact sequence 
$$0\lra \bigoplus \Q\lra \prod \Q\lra \cIi \prod \Q\lra 0,  $$
which is not split as $\prod \Q$-modules. 
We write $t=R_{10}=\cIi \prod \Q$ for brevity. 

There are two categories of  $R_{\bullet}$-modules of interest. 
\begin{defn}
(i) The {\em standard model} $\cA (\fX)$ has objects $N \lra t\tensor V \lla V$ where $N 
\lra t\tensor V$ is inverting $\cI$.  We write $(N\lra t\tensor V)$ (with round brackets) 
for brevity. 

(ii) The {\em complete model}  $\cAhat (\fX)$ has objects $\Nhat \lra \cIi \Nhat \lla V$
where $\Nhat$ is a module over $\prod \Q$ which is complete in the
sense that $\Nhat\cong \prod_se_s\Nhat$. We write $[\Nhat\lla  V]$
(with square brackets) 
for brevity. 
\end{defn}

We need one lemma about modules over $R_0$. First, we note that any
subset $U\subseteq \cF$ defines an idempotent $e_U$ with support $U$,
and we will write
$$e_GM=\colim_Ue_UM$$
where the colimit is over cofinite subsets. We note that $e_G$ is a
functor and not an element of $R_0$.
\begin{lemma}
If $M$ is an $R_0$-module with $e_FM=0$ for all $F\in \cF$ and
$e_GM=0$ then $M=0$. The corresponding statement holds in the derived category.
  \end{lemma}
  \begin{proof}
If $x\in M$ then $x$ maps to zero in $e_GM$ and hence $x$ is zero in
$e_UM$ for some cofinite set $U$. The complementary set $U^c$ is
finite so $x\in e_{U^c}M=\prod_{F\in U^c}e_FM$ and $x$ is
zero. Passage to idempotents and directed colimits are exact. 
    \end{proof}

    \begin{lemma}
      \label{lem:AisAhat}
  There is an adjoint equivalence
  $$\adjunction{\kappa}{\cA (\fX)}{\cAhat (\fX)}{\tau}.$$
  The left adjoint is defined by
  $\kappa (N\lra t\tensor V)=[\cIi \Nhat \lla V]$ where
  $\Nhat=\prod_Fe_FN$. The right adjoint is given by 
  $\tau 
  [\cIi \Nhat \lla V]=(N\lra t\tensor V)$ where $N$ is defined by the 
  pullback square 
  $$\xymatrix{
    N\rto \dto & t\tensor V\dto \\
    \Nhat \rto &\cIi \Nhat 
    }$$
  \end{lemma}
  \begin{remark}
    Note that $N$ is not the space of global sections of the
    sheaf.

    There are numerous ways of packaging the data of a
    sheaf using pullbacks. Our models occur in the pattern
    $$C(\fX)\modules \lra \cA (\fX)\lra \cAhat(\fX), $$
    with the controlling vector spaces
    $$\cF (\fX)\lra N \lra \hat{N}. $$
    Even for the unit object, these natural maps are not isomorphisms
    if $\fX$ has infinitely many points. 
  \end{remark}

\begin{proof}  Neither functor changes
  the vertex $V$, so we need only consider the nub.

  The argument for the counit is straightforward. We need only observe
  that when calculating $\tau [\cIi \Nhat\lla  V]$ the defining
  pullback square for $N$ shows $e_F\Nhat=e_FN$, so that
  $\Nhat=\prod_Fe_FN$ and $\kappa \tau =id$.

  For the unit we must calculate   $\tau \kappa (N\lra t\tensor
  V)$. It remains to show that if $N$ is given the square 
    $$\xymatrix{
    N\rto \dto & t\tensor V\dto \\
    \Nhat \rto &\cIi \Nhat 
  }$$
  formed by completing $N$ is a homotopy pullback. Since $\Nhat\lra
  \cIi\Nhat$ is surjective it suffices to prove the square is a
  pullback. The square provides a map $\nu: N\lra P$ where $P$ is the
  pullback. 

  We will observe that $e_F\nu$ is a isomorphism for all $F\in \fD_0$,
  and that $e_G\nu=\colim_Ue_U\nu$ is an isomorphism. For $e_F$ the
  statement is immediate: $e_FN=e_F\Nhat$ by definition and $e_F$ annihilates
  the right hand column, the square is a pullback after applying
  $e_F$.

  For $e_G$, we see $\Nhat=\ilim_Ae_AN$ over finite sets $A\subseteq
  \cF$, and hence the fibre is $\ilim_Ue_UN$ over cofinite sets. Thus
  $\cI$ is already invertible. 
\end{proof}

\subsection{Sheaves of rings}
We write $i:\cF=\fD_0\lra \fX$ for the open inclusion and $j:
\{G\}=\fD_1 \lra \fX$ for the closed inclusion.  In the category of sheaves over $\fX$ we consider the 
rings 
$$\mcR_1=j_*\Qu_{G} 
\mbox{  and } \mcR_0=i_*\Qu_{\cF} 
. $$
We also consider the sheaf $\mcR_{10}=j_*j^* \mcR_0$. This gives a diagram $\mcR_{\bullet}$ of rings 
$$\xymatrix{
\Qu_{\fX}  \rto \dto &\mcR_1 \dto \\
  \mcR_0\rto &\mcR_{10}
}$$
which is both a pullback and a homotopy pullback.
The Snake Lemma applied to the diagram 
$$\xymatrix{
0\rto &\bigoplus_Fi^F_*\Qu_F\rto \dto &\mcR_0\rto\dto&\mcR_{10}\rto \dto^=
&0\\
0\rto &\Qu_{\fX}\rto &\mcR_0\oplus \mcR_1\rto &\mcR_{10}\rto &0
}$$
gives the non-split short exact sequence 
$$0\lra \bigoplus_F i^F_*\Qu_F \lra \Qu_{\fX}\lra j_*\Qu_G\lra 0. $$

There are three abelian  categories of  interest. First the category
of $\Gamma (\Qu_{\fX})$-modules (which is a category of $\Q$-modules)
and then two categories of sheaves. The first is the category of
modules over  $\Qu_{\fX}$, and the second is an analogue of the
standard model. However we observe that in the category of sheaves 
the modules in the standard model are complete. 
\begin{lemma}
  \label{lem:shvF}
There are equivalences of categories 
$$\mcR_0\modules/\fX \simeq \Qu_{\cF}\modules/\cF\simeq \prod_{F\in \cF}(\Q\modules)$$
\end{lemma}

Because of the lemma we will use the complete terminology. 

\begin{defn}
The {\em complete model} $\cAhat (\fX)$ has objects $\cN \lra
\mcR_{10}\tensor \cV \lla \cV$ where $\cN$ is an $\mcR_0$-module, $\cV$
is a $\mcR_1$-module and $\cN \lra \mcR_{10}\tensor \cV$
is an isomorphism on the stalk at $G$.
\end{defn}

\begin{lemma}
  The following four abelian categories are equivalent
  \begin{itemize}
  \item The category of sheaves of $\Q$-modules over $\fX$
  \item The complete model $\cAhat(\fX)$
  \item The category of modules over $\Gamma (\Qu_{\fX})$
    \item The standard model $\cA (\fX)$.
\end{itemize}
$$\shv/\fX\simeq \cAhat(\fX)\simeq \Gamma (\Qu_{\fX})\modules\simeq \cA(\fX)$$
  \end{lemma}
  \begin{proof}
The equivalence of the first two is the classical fact that a sheaf is specified by stalks 
together with germ-spreading data. Indeed, by Lemma \ref{lem:shvF}, an $\mcR_0$-module $\cN$ is 
determined by its stalks $\cN_F$. Thus a sheaf $\cG$ is specified by 
$V=\cG_G$ and $\cN_F=\cG_F$ and the germ-spreading map 
$$ V\lra \cIi \prod_F\cN_F. $$

The equivalence with modules is given by global sections $\shv/\fX\lra
\Gamma(\Qu_{\fX})\modules$ in one direction and by the construction of
the germ spreading map from ring theoretic data in the other. Given a
module $M$ over $\Gamma(\Qu_{\fX})$ we form the sheaf $\cG$ with stalks
$\cG_F=e_FM$ and $\cG_G=\colim_{U\ni G}e_UM$ and germ-spreading map
$$\cG_G=\colim_{U\ni G}e_UM\lra \colim_{U\ni G}e_U\prod_Fe_FM=
\colim_{U\ni G}e_U\prod_F\cG_F.$$

The final equivalence is Lemma \ref{lem:AisAhat}. 
    \end{proof}

\subsection{Homological algebra}
It is both instructive and reassuring to make the homological algebra
of sheaves explicit in this case.

  \begin{lemma}
The category of sheaves over $\fX$  has injective dimension 1. 
\end{lemma}

\begin{proof}
  Skyscraper sheaves are injective. For a general sheaf $\cG$
  we may construct an injective resolution
  $$0\lra \cG \lra \cG_G \oplus \prod_{x\in \fD_0} \cG_x\lra \cG'\lra
  0$$
  and $\cG'$ is concentrated at $G$ and it is therefore
  injective.

  This sequence does not split if $\cG$ is the constant sheaf at $\Q$
  and therefore not all sheaves are injective.
\end{proof}

\begin{lemma}
  All sheaves are soft. 
  \end{lemma}

  \begin{proof}
Any closed set is either a finite set of points in $\fD_0$, the
complement of a finite set of points of $\fD_0$  or else the
limit point. If  $A\subseteq \fD_0$ is finite, then both $A$ and $A^c$ are open
and closed. Accordingly, $\cF (\fX)=\cF(A)\oplus \cF(A^c)$ and any
section over $A$ or $A^c$ extends over $\fX$.

The limit point is the intersection of clopen sets $U$ and so 
$$\cF(G)=\colim_U \cF(U)$$
is the direct limit of surjections. 
\end{proof}

\part{Component structures and equivariant sheaves}

The idea is that the structure originates as a $G$-equivariant sheaf
over the $G$-space $\sub(G)$
of subgroups: this would be a $G$-map $\cFt \lra \sub (G)$ that is
\'etale, and so that the fibre (stalk) over $K$ is a  $\Q$-vector space
$\cFt_K$. We also require the action is linear on fibres, so that $\cFt_K$
is a representations of the isotropy group $G_K=N_G(K)$; we impose the
additional  condition that $K$ acts trivially on $\cFt_K$, so that the
action is inflated from the Weyl group $W_G(K)$.  With appropriate
foundations this would be a sheaf 
over the quotient stack $[\sub(G) //G]$. However our model states 
that this sheaf descends to a $\cW$-equivariant sheaf $\cF$ on the
quotient space $\sub (G)/G$, in the sense that $\cF_{(K)}$  is
determined by a representation of $W_G(K)$ on $X_{(K)}$. In effect,
the continuous and highly redundant data of a $G$-equivariant object
has been replaced by the much tighter algebraic data of a sheaf of
representations of finite groups.

The first task in making sense of this  is that to specify $W_G(K)$ we need to select
a representative subgroup $K$ for the conjugacy class.
In \cite{gqblocks} we show  that one may select conjugacy class 
representatives. In fact we find a section of $\sub (G) \lra \sub(G)/G$, block
by block,  and compatible with containment of nearby subgroups.
For example, in the case of the dihedral block of $O(2)$ we
may take the section consisting of dihedral subgroups containing 
reflection in the $x$-axis. 

For the purpose of the present paper, we suppose this selection has
been made. This gives the space $\cV^G_H$ a component structure, and 
there is a finite group $W_G(K)$ associated to each point $(K)$ and maps
$W_G(K)\lra W_G(H)$ for all $K$ in a neighbourhood of $H$. We may then
declare that a $\cW$-sheaf over a space with a component structure is
a sheaf $X$ so that $X_K$ is a $W_G(K)$-module compatibly.

\section{Component structures on Stone spaces}
\label{sec:componstone}
The structures arising from rational $G$-spectra consist of
{\em Weyl equivariant} sheaves. In Part 1 we described sheaves over Stone spaces
and in Part 2 we describe the equivariance. In effect this consists of
the action of a finite group on each stalk of the sheaf.

\subsection{Component structures}
The notion of 
component structure codifies the structure required of this collection 
of finite groups $\cW_x$.

\begin{defn}
  For a Stone space $\fX$, a {\em component structure} consists of a finite
  group $\cW_x$ for each point $x\in \fX$, and for each $x$ there is
  a clopen neighbourhood $U_x$ of $x$ with a homomorphism $i^x_y: \cW_y\lra
  \cW_x$ for all $y\in U_x$. These are required to be transitive in
  the sense that if $z\in U_y\cap U_x$ and $y\in U_x$, we have
  $i_z^x=i_y^xi_z^y$. 
  \end{defn}

  \begin{example}
    If $H$ is a compact Lie group the space $\Phi H$ of conjugacy
    classes of closed subgroups with finite Weyl group is a Stone
    space, as is any compact subset. We will choose $\fX=\cV^G_H$ to
    be a block, and we will assume that there is a continuous
    order-preserving selection of conjugacy class representatives, in
    the sense that $\sub (G)\lra \sub(G)/G$ has a section over
    $\cV^G_H$ (it will be proved in \cite{gqblocks} that this holds
    for all blocks). 

    By the upper semicontinuity of normalizers \cite[10.3]{AVmodel},
    every subgroup $K\in \Phi H$ has a neighbourhood $U_K$ so that
    if $L\subseteq K$ and
    $(L)\in U_K$ then $N_G(L)\subseteq N_G(K)$.

    We take the component structure in which $\cW_K=\pi_0(W_G(K))$. The
    component structure is given by using the neighbourhoods $U_K$ and
    the inclusion induces a map
    $$ N_G(L)/L\lra N_G(K)/L\lra N_G(K)/K,  $$
    and applying $\pi_0$ we obtain the component structure. This is independent of choice of
    representatives; indeed, by the Montgomery-Zippin Theorem two
    representatives of the conjugacy class $L$ lying in $K$ will
    be conjugate by an element of $N_G(K)$. This means the map
    $$\cW_L=N_G(L)/L\lra N_G(K)/L\lra N_G(K)/K=\cW_K$$
    depends only on the conjugacy classes $(L)$ and $(K)$.
    
    Since the maps are induced by inclusions of normalizers, the
    transitivity condition is automatic. 
  \end{example}

  When component structures are used for sheaves of rational vector
  spaces we may package the relevant structure in a more familiar way,
  making sense of the idea that 
$$\cW=\coprod_x \cW_x\lra \fX$$ 
has local properties  dual to that of an \'etale space.

  \begin{lemma}
A component structure $\cW$ on $\fX$ gives rise to a sheaf of group 
rings $\Q\cW$ on $\fX$. The stalk over $x$ is $\Q\cW_x$ and the 
germ-spreading map 
$$\sigma: \Q\cW_x\lra \cIi_{x|U_x}\prod_{y\in U_x}\Q\cW_y$$
has components
$$\Q \cW_x\lra \Q\overline{\cW_y}\cong (\Q\cW_y)^{K(x,y)}\lra
\Q\cW_y, $$
where the first map is induced by restriction to the image
$\overline{\cW_y}=\im (\cW_y\lra \cW_x)$ and
$K(x,y)=\ker (\cW_y\lra \cW_x)$.

\end{lemma}

We have given four equivalent abelian categories: the standard model,
the category of sheaves over $\fX$ (codified as \etale\ spaces over
$\fX$), the standard model $\cAhat(\fX)$ in sheaves,
and the category of modules over $\RfX$.

Given a component structure $\cW$ on $\fX$ we will describe a category
of $\cW$-equivariant objects in each of these cases. Finally, we will observe that the
four categories are equivalent. 

\begin{thm}
  \label{thm:cWshvscatStone}
Suppose $\fX$ is a Stone space of finite Cantor-Bendixson rank with
component structure $\cW$,  the  following four abelian categories are equivalent 
\begin{itemize}
\item the category $\cW\!-\!\shv/\fX$ of  sheaves of 
$\Q$-modules. 
  \item the complete abelian category $\cAhat (\fX, \cW)$ 
\item the category of modules over the ring $\Gamma (\Qu_{\fX}[\cW] )$
  of global sections of the sheaf of Weyl algebras
\item the standard abelian category $\cA (\fX, \cW)$
\end{itemize}
 The injective dimension of the three abelian categories is equal to the 
Cantor-Bendixson rank of $\fX$. 
\end{thm}

  \subsection{Equivariant sheaves}
One may define sheaves in two ways. Either as presheaves with
patching condition or as sections of \etale\ spaces. The equivalence
of categories is given by taking sections of \etale\ spaces over an
open set, and by defining a topology on the germs by using
representatives. We want to
consider equivariant sheaves, so it is simpler to take the
\etale\ space version.

\begin{defn}
  If $\cW$ is a component structure on a topological space $\fX$, a 
  $\cW$-sheaf on $\fX$ is an \etale\ space $Y$ over $X$ with an action of 
  $\cW_x$ on $Y_x$ so that the germ-spreading map 
  $$\sigma: Y_x\lra \colim_{U\ni x}\prod_{y\in U}Y_y$$
  is equivariant in the sense that if $\sigma(a)=[(\sigma_y(a))]$ for 
  $a\in Y_x$ then
$$\sigma_y(a)\in Y_y^{K(x,y)}, \mbox{ where } K(x,y)=\ker (\cW_y\lra
\cW_x). $$
(Thus $\sigma_y(i^x_y(g_y)a)=g_y\sigma_y(a) $ for $g_y\in \cW_y$).
\end{defn}

\begin{remark}
  In the group case we have $\cW_K=N_G(K)/K$ and so
  $$K(H,K)=\ker
  (W_G(K)\lra W_K(H))=(N_G(K)\cap H)/K=N_H(K)/K=W_H(K). $$
  This in turn means that the condition is
  $$\sigma_H : Y_H \lra \cIi_H\prod_K Y_K^{W_H(K)}. $$
  \end{remark}

\begin{lemma}
  Morphisms in the category $\cWshv/\fX\lra \shv/\fX$ of $\cW$-sheaves of
  $\Q$-vector spaces on $\fX$ are maps of sheaves that commute with
  the Weyl actions. 

  The category  is equivalent to the category of sheaves of modules over the
  sheaf of Weyl group rings $\Q [\cW]$.
\end{lemma}

\begin{proof}
We need only verify that the stalkwise equivariance condition
automatically gives a map $\Qu_{\fX}[\cW]\times Y \lra Y$ of \'etale
spaces. In other words we must show it is continuous. If $\sigma $ is
a local section of $Y$ over $U\subseteq \fX$, the inverse image is a union of sets 
$$(g,\sigma):=\{ g_y^{-1}\sigma(u)\st u\in U\}, $$
where $g$ is defined on a neighbourhood of $x$, and each of the sets
$(g,\sigma)$ is open. 
\end{proof}

This also gives sense to the standard model $\cAhat(\fX, \cW)$. The
sheaves $\mcR_A$ are all equivariant sheaves (with the trivial action)
and we take modules in the category of $\cW$-equivariant sheaves.

\subsection{Equivariant modules over $\RfX$}

For categories of modules, the definition is straightforward. 
We have $\RfX=\Gamma(\Qu_{\fX})$ and if $\cW$ is a component
structure on $\fX$,  we take $\RcWfX=\Gamma(\Qu_{\fX}[\cW])$. 
By definition, a $\cW$-equivariant $\RfX$ module is an
$\RcWfX$-module.

Fortunately, we have already established the framework for dealing
wiith the delicate matters. 

\begin{lemma}
  \label{lem:RcWfXPB}
$\RcWfX$ itself is the pullback and  homotopy pullback of the diagram 
$R_A[\cW]$ of rings. 
\end{lemma}

\begin{proof}
  It suffices to prove the statement for modules, since the maps
  in the diagram are ring maps.

 We fall back to a sheaf theoretic argument. We see that the sheaf
 $\mcR[\cW]=\Qu_{\fX}[\cW]$ is the pullback and homotopy pullback of the the
 sheaves $\mcR_A[\cW]$ defined as in Subsection \ref{subsec:lifttoshv}. We now take global sections and use the fact
 that sheaves are soft to obtain a pullback and homotopy pullback of
 global sections $\Gamma (\mcR_A[\cW])$.

 Finally, we calculate
 $$\Gamma (\mcR_A[\cW])=R_A[\cW] $$
as in Lemma \ref{lem:GammaRAisRA}. 
\end{proof}

\begin{remark}
  An alternative approach is to start from
  the fact that $\RfX$ is the homotopy pullback of the
  rings $R_A$ (Lemma \ref{lem:ringcubeexact}), and take the tensor product with $\RcWfX$.  The argument may be
  completed by showing that $R_A\tensor_{\RfX}\RcWfX=R_A[\cW]$. We
  will establish this for other purposes in Lemma \ref{lem:RAWsmall} below.
\end{remark}

\subsection{The equivariant standard model}
When $A=\{a_0>\cdots >a_s\}$, the ring $R_A$ has idempotents $e_{H_s}$
for each $H_s\in \fD_{a_s}$. A $\cW$-equivariant $R_A$-module $M$ is an
$R_A$-module with actions of $\cW_{H_s}$ on $e_{H_s}M$.  A
$\cW$-equivariant $R_{\bullet}$-module is an $R_{\bullet}$-module $M$ so
that each $M_A $ is a $\cW$-equivariant $R_A$-module, and so that if
$B=\{a_0>\cdots >a_s>b\}$, the  map
$$M_A\lra M_B$$
has the property that
$$e_{H_s}M_A\lra e_{H_{s+1}}M_{B}$$
is equivariant along the map $\cW_{H_{s+1}}\lra \cW_{H_s}$
when $H_{s+1}$ is close enough to $H_s$.

\subsection{Equivalences}
\label{subsec:equiv}
We consider the diagram of adjoint equivalences in Section
\ref{sec:stalkstandard}. We note that the functors (but not the equivalences)
extend to  the entire category of modules over 
the diagrams $R_{\bullet}$ (of rings) or $\mcR_{\bullet}$ (of sheaves 
of rings). In all cases the functors respect the actions of the Weyl
groups.  We therefore obtain another diagram of equivalences 
$$\xymatrix{
V\mbox{-model}= \cAhat(\fX, \cW) \rto^(0.6){\simeq}&\cW-\shv/\fX \ar@{=}[r]&\Qu[\cW]\modules\rto^{\simeq}\dto^{\simeq}&
 R[\cW]\modules\dto^{\simeq}\\
 &&\ilim^w_A \mcR_A [\cW]\rto^{*\simeq*} \dto^{\simeq}&\ilim^w_A R_A[\cW]\dto^{\simeq} \\
 &&DG- \cAhyb(\fX, \cW)\rto^{*\simeq*} &DG-\cA (\fX,\cW) \ar@{=}[d]\\
&&& N \mbox{-model}
 }$$

\subsection{Injective dimension}
We note here that by the usual averaging argument, equivariance does
not change the injective dimension of these abelian categories. 

\begin{prop}
The abelian category $\cW\!-\!\shv/\fX$ of equivariant sheaves on
$\fX$ is of injective dimension equal to the Cantor-Bendixson rank of $\fX$. 
\end{prop}

\begin{proof}
For any $\Qu_{\fX}[\cW]$-module $M$, we may construct an injective
resolution $0\lra M\lra I_0\lra I_1\lra I_2\lra \cdots$. By
restriction we may view it as a complex of injective
$\Qu_{\fX}$-modules, and as such it terminates in the sense that the
image of $I_n\lra I_{n+1}$ has a complement for some $n\leq \mathrm{CB-rank}(\fX)$.
Thus there is a map $f: I_{n+1}\lra I_n$ of sheaves that is the
identity on the image. 

We now apply the averaging map 
$$\pi : \Hom(I_{n+1}, I_n)\lra \Hom(I_{n+1}, I_n)^{\cW}$$
defined by 
$$\pi (f)((v,x))=\frac{1}{|\cW_x|}\sum_{w\in \cW_x}wf_x(w^{-1}v) $$
to obtain a $\cW$-equivariant map still the identity on the image. 
\end{proof}

\section{Generators of equivariant module categories}
\label{sec:gens}
When considering a finite group $F$, it is valuable to know that
the category of $\Q [F]$-modules is generated by the object $\Q
[F]$-itself. If we consider the category of modules over a finite
product of such rings, perhaps viewed as the category of sheaves of
modules over a sheaf of rings, there is nothing more to say. However, 
when there are infinitely many, as in component structures, there are
 continuity conditions on the groups $F$ and on
the modules. Our purpose here is to elucidate this. 

We now have three  different but equivalent abelian categories of equivariant
objects: sheaves, modules and diagrams. There are obvious generators
in the first two categories, and this gives us a generator in the
third.

     \subsection{Modules}
        The most elementary form of the model is that we are discussing 
        modules over the ring $\RfX[\cW]:=\Gamma(\Qu_{\fX}[\cW])$. The statement about 
      generators is routine in this context.

      \begin{lemma}
\label{lem:RXWgens}
The abelian category of $\RfX [\cW]$-modules is generated by 
$\RfX [\cW]$, and hence the derived category is generated as 
triangulated category by $\RfX[\cW]$. 
\end{lemma}

\subsection{Sheaves}

The most approachable  form of the model is that we are discussing 
modules over the sheaf $\Qu_{\fX}[\cW]$ of rings. We will sketch a
direct proof, but it follows from Lemma \ref{lem:RXWgens} and the equivalence of categories
established in the Section \ref{sec:componstone}. 

\begin{lemma}
The abelian category of sheaves of $\Qu_{\fX}[\cW]$-modules is generated by 
$\Qu_{\fX} [\cW]$, and hence the derived category of sheaves of modules is generated as 
triangulated category by $\Qu_{\fX}[\cW]$. 
  \end{lemma}
\begin{warning}
The objects $\Qu$ and $\Qu[\cW]$ are not small if $A\neq \emptyset$.
\end{warning}

  \begin{proof}
    It suffices to show that for an arbitrary module $M$, 
we can construct an epimorphism from a  sum of copies of
$\Qu_{\fX}[\cW]$, and for this it suffices to construct a map which is
an  epimorphism on stalks.
   
 If $x\in M_K$,  then we may define  $f_K: \Q
    [W_G(K)]\lra M_K$ by taking $f_K(1)=x$. From the sheaf condition,
    $x\in M_K$ is represented in an open and closed neighbourhood
    $U_K$ of $K$, and by choosing a smaller neighbourhood if
    necessary, $N_G(K')\subseteq N_G(K)$. From the equivariance
    condition it is consistent to extend $f$ over $U_K$ by
    $f(K')=x(K')$. Then $f$ can be extended by zero on the complement
    of $U_K$.
        \end{proof}

 \subsection{The standard model}

 The model which makes the link with topology is the standard
 model. From the equivalence of abelian categories it is now clear it
 has a single generator which consists of the image of the object
 $\Gamma (\Qu_{\fX}[\cW])$. The main  purpose of Section \ref{sec:gens} is to
 make this explicit.

The constant sheaf $\Qu_{\fX}$ is the diagram of rings $R_A$ as in
Subsection \ref{subsec:splicingrings}. By Lemma \ref{lem:RcWfXPB} the image
of the generator $\Qu_{\fX}[\cW]$ is 
$$R_A[\cW]=\prod_{\height(H_0)=a_0}\Q [W_G(H_0)]\tensor \cIi_{H_0|a_1}
\prod_{\height(H_1)=a_1} \cIi_{H_1|a_2} \cdots 
\prod_{\height(H_{s-1})=a_{s-1}}\cIi_{H_{s-1}|a_s}\prod_{\height(H_s)=a_s}
\Q  $$
formally the same as for  the constant sheaf $\Qu_{\fX}$ except that the final $\Q$ in the
$H_s$ factor should be replaced by $\Q[W_G(H_s)]$. The generation
statement then follows from the equivalence of Section
\ref{sec:componstone} and  Lemma \ref{lem:RXWgens}.

\begin{lemma}
  \label{lem:AGWgens}
  The abelian category $\cA (\fX, \cW)$ is generated by the
diagram $R_{\bullet}[\cW]$, with value $R_A[\cW]$ at $A$. \qqed
  \end{lemma}

\part{Algebraic models of Weyl-finite rational $G$-spectra}

Blocks of subgroups with finite Weyl groups are particularly simple, but there are
still issues to be dealt with that have not arisen for tori or in low
dimension.

As in other cases, the framework of the proof is given by two  adelic pullback
cubes, one in topology and one in algebra. The one in topology shows
one can assemble the category of $G$-spectra from simple pieces,
and the one in algebra  shows that one can assemble the algebraic
model from simple pieces.  We need only check that the simple pieces
in topology and algebra are equivalent. For blocks with finite Weyl
groups we saw in Part 1 that
the adelic cube describes how to build a sheaf on a filtered space
from its stalks when the pure strata are discrete and in Part 2 we
upgraded this to take account of the component structure. We
described the algebra of this in detail, and in Section \ref{sec:Rtop} we
show that the pullback cube can be realised in homotopy theory where
it is also a pullback. Parts 1 and 2 showed that the entire algebraic
model can be recovered from the stalks and the component structure.  
In Part 3, we show the same is true on the topological side.

\section{The topological pullback cube}
\label{sec:Rtop}
The diagram of rings $R_A$ from Section \ref{sec:ringcubes}
is both a pullback and homotopy pullback. The strategy is to construct
a counterpart homotopy pullback in the world of  naive-commutative
ring $G$-spectra. When  the component structure is trivial, this is
the complete argument. In general a further argument is necessary to
identify generators and their endomorphism rings. 

In the present section we will construct the diagram of commutative
ring spectra, and in Section \ref{sec:topgens} we will discuss
generators.

  \subsection{The cube of ring spectra}
  We define naive-commutative ring spectra $R_A^{top}$.

  We begin with the stalks at a subgroup $K$. By hypothesis, we are
  working in a clopen set $\fX$ in which all subgroups have finite
  Weyl groups. Accordingly $K$ has a system of clopen neighbourhoods
  $U_{\alpha}$ so that $\bigcap_{\alpha}  U_{\alpha}
  =\{K\}$. Associated to each clopen  $U$ there is
  an idempotent $e_U \in [S^0,S^0]^G$ with support $U$.
  \begin{lemma}
    The map
    $$X\lra L_KX=\colim_U e_UX$$
    is smashing Bousfield localization away from $G$-spectra $X$ with
    $\Phi^KX\simeq 0$. Accordingly
    $$S^0\lra S^0_K$$
    may be realized as a map of naive-commutative ring spectra. 
    \end{lemma}

    \begin{proof}
The map $S^0\lra e_U S^0$ is smashing localization away from objects
supported on $U^c$ and the map $S^0\lra S^0_K$ is smashing
localization away from objects with $\Phi^KT\simeq 0$.
      \end{proof}

  \begin{defn}
Given  $A=\{a_0>a_1>\cdots 
  >a_s\} \subseteq [r]$, we may  construct a sheaf $\Rtop_A$ on $\fX$
  as follows. First, if $s=0$ we  start with  the commutative ring
  $G$-spectra
  $$\Rtop_K=S^0_K .$$

 Next for $a_s\in [0,r]$, we take 
  $$\Rtop_{a_s}=\prod_{\height(K_s)=a_s}\Rtop_{K_s}. $$
  Now if 
  $d_0A=\{a_1>\cdots >a_s\}$ and $\mcR_{d_0A}$ is defined 
  we take 
  $$\Rtop_A=\prod_{\height(K_0)=a_0}\cIi_{K_0|a_1} \Rtop_{d_0A}. $$
  By construction there is a map 
  $$\Rtop_{d_0A}\lra \mcR_A.  $$
  Altogether we have 
  $$\Rtop_A=\prod_{\height(K_0)=a_0}\cIi_{K_0|a_1}\prod_{\height(K_1)=a_1}\cdots
  \cIi_{K_{s-1}|a_s}\prod_{\height(K_s)=a_s}S^0_{K_s} . $$
\end{defn}

      \subsection{The cube is a pullback}
\newcommand{\PBS}{S^0_{PB}}
We have constructed a cube $\Rtop$, with 
$\Rtop_{\emptyset}=S^0_{\fX}$.

\begin{lemma}
  \label{lem:htpyRtop}
Taking homotopy groups gives the corresponding rings 
   $$\pi^G_*(\Rtop_A)=R_A, $$
   and this is natural for the structure maps in the cube. 
 \end{lemma}

 \begin{proof}
It is clear if $A=a_s$ by construction. 
  We then repeatedly use the fact that homotopy preserves products and
  for a module $M(x)$ over $\prod \Q$,  
$$\piG_*(\colim_{U\ni x}e_U M(x))=\colim_{U\ni x}e_U\piG_*(M(x)).$$
      \end{proof}

      \begin{cor}
        \label{cor:toppullback}
  The cube $\Rtop_\bullet$ is a homotopy pullback. 
  \end{cor}

We give two proofs. The first is more direct, but the second embodies
the use of algebraic models more effectively.
  
  \begin{proof} (Strategy 1)
    We have a homotopy pullback cube in topology, which may be proved
    by the method of \cite{adelicm}.   
  \end{proof}

      \begin{proof}    (Strategy 2)
        We form the homotopy pullback $\PBS$ of the punctured cube. 
By construction we  obtain a map $S^0\lra \PBS$ and we need to know
this is an equivalence. The pullback cube gives a spectral
        sequence for $\piG_*(\PBS)$, and by Lemma \ref{lem:htpyRtop}
            it starts from the adelic cohomology \cite{adelicc}. By Lemma \ref{lem:ringcubeexact} the 
    spectral sequence collapses to give the homotopy of $\PBS$, which
    is identified as $\Gamma (\Qu_{\fX})$. By tom Dieck's calculation
    this is $\piG_*(S^0)$.

    We now need to repeat the argument for each subgroups $H$ of
    $G$. Of course we have a map $\sub(H)/H\lra \sub(G)/G$, and the
    inverse image of $\fX$ in $\sub(H)/H$ is a clopen set $\fY$
    dominated by $H$. We may assume by induction on the dimension and
    number of components that the result for $\fY$ has already been
    proved.
    
    However the map $\fY\lra \fX$ will not be  a bijection. If we let
    $\fX|H$ denote the image, it is clear that the parts of the
    diagram corresponding to $\fX\setminus \fX|H$ are all trivial when
    restricted to $H$. The fibre of $\fY\lra \fX|H$ over $K\subseteq
    H$ is the finite set $N_G(K)/N_H(K)$. There are maps
    $$R_{\fX|H}\stackrel{\Delta} \lra R_{\fY}\lra R_{\fX|H}$$
where the first is diagonal and the second is fusion; the composite is
multiplication by $[N_G(K):N_H(K)]$ on $K$, and hence is an
equivalence. We infer that since $S^0\lra (\PBS)_H$ is a $\pi^H_*$
isomorphism so is $S^0\lra (\PBS)_G$. 
  \end{proof}

As in \cite[4.1]{diagrams} and \cite[9.5]{adelicm}, this shows the
category of $G$-spectra can be reconstructed from the diagram of module categories. 
      \begin{cor}
        \label{cor:wtdPB}
  $$\Gspectra/\fX\simeq \wilim_A \Rtop_A\modules. \qqed $$
  \end{cor}

  \section{Component structures}
  \label{sec:topgens}

We observed in Lemma \ref{lem:RcWfXPB} that the ring $\RfX [\cW]$ is a pullback of the rings
$R_A[\cW]$, and hence 
  $$\Q[\cW]\modules \simeq \wilim_A \Rtop_A [\cW] \modules. $$
Similarly, Corollary \ref{cor:wtdPB} we have seen 
 $$\Gspectra\lr{\fX}\simeq \wilim_A \Rtop_A\modules\!-\! \Gspectra. $$
It remains to establish the equivalence
$$\Rtop_A\modules\!-\!\Gspectra\simeq \Rtop_A[\cW]\modules,   $$
in a way that is compatible with the cubical diagram.

Forgetting the diagram for a moment, we do it on each object by
identifying a single small generator $g_A$ of $\Rtop_A$-module
$G$-spectra, and showing its endomorphism ring has homotopy $R_A[\cW]$
in degree 0. By Shipley's Theorem and the formality of DGAs with
homotopy in degree 0, this completes the argument.

As it stands, this does not give a map of punctured cubical
diagrams. The generator $g_A$ at $A$ does give a small generator
$\hat{g}_B=R_B\tensor_{R_A}g_A$ at $B$. Moreover, since 
$g_A$  is small over $R_A$, $\hat{g}_B$ is small over $R_B$ and  we may use it as a bimodule for a Morita
equivalence and obtain a commutative diagram.

\subsection{Stalkwise}
We begin with the case of singleton isotropy, where the procedure is
routine.
          
          \begin{lemma}
            \label{lem:Rtopgen}
If $\Rtop_K$ is a ring $G$-spectrum with geometric isotropy concentrated
at $K$ where $W_G(K)$ is finite
then the category of  $\Rtop_K$-modules is generated by $G/K_+\sm 
         \Rtop_K$. 
        \end{lemma}
        \begin{proof}
By construction, the category of $\Rtop_K$-module $G$-spectra
          is generated by the objects 
$G/H_+\sm \Rtop_K $ as $H$ runs through subgroups of $G$. 
It remains to show that  each of these can be built from the one with $H=K$. 

First, by geometric isotropy,  the only  terms $G/H_+\sm \Rtop_K$ which are non-trivial 
are  those with $K$ conjugate to a subgroup of $H$. Suppose then that $K\subseteq 
H$.

The composite $H/H_+\lra H/K_+\lra H/H_+$ is multiplication by the
Burnside ring element $[H/K]$. When applied to an object with
geometric isotropy $K$ this is multiplication by $|(H/K)^K|=W_H(K)$;
since it is non-zero it is a unit in $\Q$.  
        \end{proof}

          \subsection{Vertexwise}
          We now pick a non-empty height flag $A=\{a_0>a_1>\cdots >a_s\}$ and 
assemble the contributions to the      category of modules over $\Rtop_A$.

\begin{lemma}
  \label{lem:RAWgen}
                   The category of $\Rtop_A$-module $G$-spectra
                   is generated by the image of the single object 
                   $$\Rtop_A[G/\bullet_+]:=
\prod_{\height(H_0)=a_0} (G/H_0)_+\sm \cIi_{H_0|a_1}
\prod_{\height (H_1)=a_1} \cIi_{H_1|a_2}     \cdots 
\prod_{\height (H_{s-1})=a_{s-1}}\cIi_{H_{s-1}|a_s}
\prod_{\height(H_s)=a_s}\Rtop_{H_s}.  $$
The same applies if the cell $(G/H_0)_+$ in the 0th product is
replaced by $(G/H_i)_+$ in the $i$th product. 
                  \end{lemma}    
                 \begin{proof}
We have argued in Lemma \ref{lem:Rtopgen} that there $G/H_+\sm \Rtop_K$ is either trivial (if 
$K$ is not subconjugate to $H$) or a retract of $G/K_+\sm 
\Rtop_K$. Applying this at each point gives the required argument. 
\end{proof}

Now we have a single generator, it is natural to consider the 
endomorphism ring spectrum 
$$\cE_A=\Hom_{\Rtop_A}(\Rtop_A[G/\bullet_+], \Rtop_A[G/\bullet_+]).$$

\begin{prop}
\label{prop:piGcEA}
$$\piG_*(\cE_A)=R_A[\cW].$$
\end{prop}

\begin{proof}
  First of all, there is the calculation (when  $H_0$ has finite Weyl group) 
  $$[\cIi_{H_0} (G/H_0)_+, [\cIi_{H_0} (G/H_0)_+]^G=[S^0,
  \Phi^{H_0}G/H_0]^{H_0}=\Q W_G(H_0). $$

  The main steps  are formal calculations applied iteratively. 

Argument 1: we make an 
argument of the following form, for suitable $G$-spectra  $P_i$
$$[\prod_i P_i, \prod_j P_j]^G=\prod_j [\prod_i P_i, P_j]^G=\prod_j [P_j, 
P_j]^G; $$
the first equality is the universal property of the product. The 
second involves showing $[\prod_{i\neq j}P_i, P_j]^G=0$.

Argument 2: we make an 
argument of the following form, for suitable $G$-spectra  $P_i$
$$[\prod_i \cIi_i P_i, \prod_j \cIi_j P_j]^G=\prod_j [\prod_i \cIi_i
P_i, \cIi_j P_j]^G=\prod_j \cIi_j [P_j, 
P_j]^G; $$
the first equality is the universal property of the product. The 
second begins as before by showing $[\prod_{i\neq j}\cIi_i P_i, \cIi_j
P_j]^G=0$, as in Argument 1,  and then 
$$[\cIi_j P_j, \cIi_j P_j]^G=[P_j, \cIi_j P_j]^G=\cIi_j [P_j,
P_j]^G.$$
Here the first step is the universal property of the localization
$\cIi_j$ and the second is a smallness statement (Lemma
\ref{lem:RAWsmall}
for $P_j$.

When $A$ is of length 0, we need only apply Argument 1 with
$P_a=R_a$. The justification is that since all the terms are of the
same height $a_0$, there is a neighbourhood of each element $j$ not
containing any $i$ and hence $[\prod_i P_i, P_j]^G=0$. This justifiation
applies in all cases where we use this argument.

When $A=(a_0>a_1)$ we have
$$\begin{array}{rcl}
[\prod_{i_0}\cIi_{i_0}\prod_{i_1}R_{i_1}, 
    \prod_{j_0}\cIi_{j_0}\prod_{j_1}R_{j_1}]^G&=&
\prod_{j_0} [\prod_{i_0}\cIi_{i_0}\prod_{i_1}R_{i_1}, \cIi_{j_0}\prod_{j_1}R_{j_1}]^G\\
&=&
\prod_{j_0} [\prod_{i_1}R_{i_1}, \cIi_{j_0}\prod_{j_1}R_{j_1}]^G\\
&=&
\prod_{j_0} \cIi_{j_0}[\prod_{i_1}R_{i_1}, \prod_{j_1}R_{j_1}]^G\\
&=&
\prod_{j_0} \cIi_{j_0}\prod_{j_1} [\prod_{i_1}R_{i_1}, R_{j_1}]^G\\
&=&
\prod_{j_0} \cIi_{j_0}\prod_{j_1} [R_{j_1}, R_{j_1}]^G
\end{array}$$
We note here that the smallness of $\prod_{i_1}R_{i_1}$ involves
working with maps of modules over that ring. 
\end{proof}

\begin{lemma}
  \label{lem:RAWsmall}
If $A\neq \emptyset$ then  $R_A[\cW]$ is a small $R_A$-module.

Consequently, the object $\Rtop_A[G/\bullet_+]$ is a small $\Rtop_A$-module. 
\end{lemma}

\begin{remark}
  There are delicate issues here. We need to steer a path between
  two different perils. 

The first is quite familiar, since the constant sheaf $\Qu$ is
generally not  small unless $\fX$ is compact. Because of this, we do
not work with sheaves when $A\neq \emptyset$, but the residue of this
problem also threatens us from the other direction. 

We note that if $A=\emptyset$ we have the ring of global sections of
$\Qu [\cW]$. This is typically not small over $\Q_{\fX}$. For
example if we take the dihedral block in $O(2)$ we have stalks $\Q[W]$
over each finite dihedral group $D_{2n}$ and $\Q$ over $O(2)$. We see
that the constant sheaf splits off the Weyl group ring sheaf and hence
$$\Q_{\fX}[\cW]=\Q_{\fX}\oplus \bigoplus_n \Q_n$$
where $\Q_n$ is the skyscraper sheaf at $D_{2n}$. The maps into
$\bigoplus_n \Q_n$ contain the uncountable summand $\prod_n \Q$.
  \end{remark}

  \begin{proof}
    It suffices to prove the algebraic statement. 
    
  The essential ingredient is the fact that there is a bound on the 
order of the Weyl groups $W_{G}(H_0)$. We may therefore partition the 
subgroups $H_0$ into finitely many sets $A_0^1, \ldots , A_0^N$ so 
that those in $A_0^{s}$ have a Weyl group of order $s$. We then note 
that as $R_A$-modules 
$$R_A[\cW]=\bigoplus_{i=0}^N e_{A_0^i}R_A[\cW]=\bigoplus_{i=0}^N 
(e_{A_0^i}R_A)^{\oplus i}. $$
  \end{proof}

                   \begin{lemma}
                     \label{lem:RtopAmodel}
          There is an equivalence of categories 
          $$\Rtop_A\mbox{-mod-$G$-spectra} \simeq R_A[\cW]\mbox{-mod}. $$
        \end{lemma}

        \begin{proof}
          Since the category of
          $\Rttop_A$-modules is generated by $\Rttop_A[G/\bullet_+]$,
          Morita theory gives an equivalence between $\Rttop_A$-modules in $G$-spectra
          and modules over the endomorphism ring spectrum $\cE_A$

By Proposition \ref{prop:piGcEA}, $\pi^G_*(\cE)=R_A[\cW]$; by Shipley's
          Theorem this is equivalent to modules over a DGA. Since its homotopy
          groups are entirely in degree 0, it is formal and equivalent
          to the ring $R_A[\cW]$.
          \end{proof}

          \subsection{The cubical diagram}
   The cubical diagram is built with the extension of scalars
          map $i_A^B: R_A\lra R_B$. Applying Morita theory we have
          $R_A\modules \simeq \End(g_A)\modules$, where $g_A$ is the chosen small
          generator of Lemma \ref{lem:RAWgen}. It remains to show that
          the Morita equivalence is compatible with extension of
          scalars.

          First we recall that the Morita equivalence is the adjoint equivalence
          $$\adjunction{T_{g_A}
          }{\End(g_A)\modules}
          {R_A\modules}{E(g_A)}$$
          where $T_A(X)=X\tensor_{\End(g_A)}g_A$ and
          $E(g_A)(M)=\Hom_{R_A}(g_A,M)$
          (an equivalence because $g_A$ is (a) small and (b) a
          generator).

          \begin{lemma}
            \label{lem:cube}
There is a commutative diagram
          $$\xymatrix{
            R_A\modules \rto^{(i_A^B)_*} \dto^{\simeq}_{E(g_A)}
           & R_B\modules \dto_{\simeq}^{E(g_B)} \\
            \End(g_A)\modules \rto^{\theta} & \End(g_B)\modules
          }$$
          where $\theta : \End(g_A)-modules \lra \End(g_B)-modules $
          is defined by
          $$\theta (X)=(R_B\tensor_{R_A} X) \tensor_{\End
            (R_B\tensor_{R_A}g_A)} \Hom(g_B, R_B\tensor_{R_A}g_A). $$
where the verticals are right adjoints. The corresponding diagram with
left adjoints also commutes.       
\end{lemma}

\begin{proof}
By Lemma \ref{lem:RAWgen}  $g_A$ is a generator of the category of
$R_A$ modules, and the objects $g_B$ and $R_B\tensor_{R_A}g_A$ are
generators of $R_B$-modules. 
  
              For the diagram with right adjoints, what we need is to
              compare the image of $M$ under the left-hand route
              (namely $\Hom_{R_B}(g_B, R_B\tensor_{R_A}M)$) with the
              image under the right hand route
              $$(\Hom_{R_A}(g_A, M)\tensor_{R_A}R_B)\tensor_{\End 
                (R_B\tensor_{R_A}g_A)}\Hom_{R_B}(g_B, 
              R_B\tensor_{R_A}g_A).  $$
              Evaluation gives the natural equivalence
              $$
              \Hom_{R_A}(R_B\tensor_{R_A} g_A, R_B\tensor{R_A}M)\tensor_{\End 
                (R_B\tensor_{R_A}g_A)}\Hom_{R_B}(g_B, R_B\tensor_{R_A}g_A)
              \lra \Hom_{R_B}(g_B, R_B\tensor_{R_A}M)$$

With left adjoints, the image of an $\End(g_A)$-module $X$ via the
left hand route is $R_B\tensor_{R_A}(X\tensor_{\End(g_A)} g_A)$ and via
the left hand route it is
$$(X\tensor_{R_A}R_B)\tensor_{\End(R_B\tensor_{R_A}g_A)} \Hom(g_B, R_B\tensor_{R_A}g_A)
\tensor_{\End(g_B)} g_B , $$
again compared by a natural equivalence. 
\end{proof}

          \subsection{Assembly}
It remains only to assemble the appropriate conclusion to the
equivalences we have obtained at each vertex.

\begin{thm}
                   There is an equivalence
                   $$\Gspectra\lr{\fX} \simeq DG-\cA (\fX, \cW)$$
 \end{thm}

                 \begin{proof}
 By Corollary \ref{cor:toppullback}, the sphere $S^0_{\fX}$ is the homotopy pullback
 of the diagram $\Rtop_{\bullet}$, so that the category of
$G$-spectra over $\fX$  is the homotopical standard model (cocartesian diagrams of
modules over $\Rtop_{\bullet}$).

By Lemma  \ref{lem:RtopAmodel}
each of the categories  $\Rtop_A$-modules in the topological diagram is equivalent to the
corresponding module category  $R_A[\cW]$-modules in the algebraic
diagram and by Lemma \ref{lem:cube} the whole diagrams are
equivalent. 

Finally from Lemma  \ref{lem:AGWgens} 
in Part 2, this is equivalent to the category of 
$\RcWfX$-modules, and by  Subsection 
\ref{subsec:equiv}, this is  the
diagram of $\cW$-equivariant sheaves over $\fX$.
                   \end{proof}

\section{Examples of Weyl-finite components}
\label{sec:examples}
Our purpose here  is just  to demonstrate that examples of Weyl-finite
blocks $\cV^G_H$ are both ubiquitous and under full calculational
control. Accordingly we will focus attention on the block
dominated by the group itself (i.e., $H=G$), and the special case when the identity 
component is a torus $G_e=T$ with component group $\pi_0(G)=W$.
This is the case studied in \cite{t2wqalg}, and we will freely use the
results from that paper. 
 
The basic invariant is the integral representation $\Lambda_0=H_1(T)$,
or equivalently its dual $\Lambda^0=H^1(T)=\Hom (\Lambda_0, \Z)$.
The Cantor-Bendixson rank is the number of simple summands in the
rational representation $\Lambda_0\tensor \Q=H_1(T; \Q)$.  
If the trivial representation does not occur, then the block is
Weyl-finite. 

In fact all such examples actually occur, since we may reverse the
calculation above to let us construct examples from  the integral 
representation theory of $W$. If $\Lambda_0$ is a copy of $\Z^r$ with
a linear action of $W$ (i.e., an integral representation of $W$), then we
may take $T$ to be the Pontrjagin dual $T=\Hom (\Lambda^0, \T)$, and use
the semidirect product  $G=H=T\sdr W$. Thus $H_1(T; \Q) $ is the rationalization 
$\Lambda_0\tensor \Q$, and the Cantor-Bendixson rank
is the number of simple summands in the rationalization of 
$\Lambda_0$.

\subsection{1-dimensional examples}
We continue to assume $G$ has identity component a torus $T$ and
component group $W=\pi_0(G)$. If the action of
$W$ on $H_1(T; \Q)$ is a non-trivial simple representation then
the block of full subgroups is a 1-dimensional Weyl-finite block.

The full subgroups of $G$ (those meeting all components of $G$) form a
block. For each $W$-submodule  $\Lambda'$ of $\Lambda^0$, there are finitely
many conjugacy classes (in bijection to $H^2(W;(\Lambda')^*)$), so
altogether there is a countably infinite set $\fD_0$ of finite
subgroups and $G$ is the only height 1 full subgroup.

Amongst the first few naturally occurring examples are the following.

\begin{example} (a) The group $G=O(2)$ with $W$ of order 2 give the block
  consisting of $G$ and the dihedral  subgroups, with
  $\Lambda_0=\Zt$. With the same $W$ and $\Lambda_0$, and minor
  elaboration for non-split extensions and non-toral subgroups, we may also take
  $G=SO(3), Pin (2)$ or $SU(2)$. 

  (b) Take $G$ to be the normalizer of the maximal torus $ST(3)$  in $SU(3)$
  and $W$ to be the Weyl group $\Sigma_3$. In this case $\Lambda_0=H_1(ST(3))$
  is the reduced regular representation of $\Sigma_3$. Again with same
  $W$ and $\Lambda_0$ we may take $G=SU(3)$ or $PSU(3)$.
\end{example}

\subsection{2 -dimensional Weyl-finite blocks}
If $G$ is a toral group with 2-dimensional Weyl-finite block
$\cV^G_G$ then  $\Lambda_0\tensor \Q=S_1\oplus S_2$ for non-trivial
simple $\Q W$-modules $S_1, S_2$. There are really two classes of example. The
simplest is the mixed case when $S_1$ and $S_2$ are distinct, and the
more complicated is the pure case when $S_1\cong S_2$.

In the mixed case, a finite index submodule of $\Lambda_0$ splits as a
sum corresponding to $S_1\oplus S_2$, 
so there is a finite cover of $G$ whose identity component $T$
splits as a product of two tori with $W$-action. The height 1 subgroups then fall 
into two types  corresponding to the two factors. Each of the two
types falls into finitely many one-parameter families. 

For an example of a 2-dimensional mixed block,
one may use the normalizer of the maximal torus in $U(2)$, which 
was treated in \cite{t2wqmixed}.    The Weyl groups are not finite in
that case, but the classification of subgroups does not depend on
this, so it  illustrates the shape of 2-dimensional Weyl-finite blocks. 
We take the maximal torus $T$ in $U(2)$ to consist of diagonal
matrices.  This is the sum of the centre (scalar matrices) and
the anti-centre ($\diag (z,\overline{z})$ for $z\in \T$), which intersect in a
subgroup of order 2. The height 1 subgroups fall into two families.
Those containing the centre form one sequence, and those containing 
the anticentre  form two sequences. (These are simply sequences because 
of the obvious parametrisation of rank 1 subgroups of $\Z$).

Similarly, the pure case is like the 2-torus from
ancient times, but we will give more detail to show 
the additional structure.

\subsection{A pure 2-dimensional Weyl-finite block}
We consider the first example where the two rational 
representations are isomorphic to illustrate the complexity of the
pure case. We take $W$ to be of order 2, and  $\Lambda^0=\Zt\oplus \Zt$.

Up to group isomorphism, there are only two toral groups with this data:
one is  the split extension and the other is a non-split 
extension (the three distinct non-split extensions are  isomorphic as groups). 

As usual, subgroups $S\subseteq T$ correspond to subgroups
$\Lambda^S=(T/S)^*\subseteq T^*=\Lambda^0$.
Because $W$ acts as negation on $\Lambda^0$,  there is no restriction
on the lattice $\Lambda^S$. Furthermore $W$ acts as negation on $\Lambda^S$ and
$H^2(W; \Lambda_S)=0$. 

\subsection{The space of subgroups (split extension)} If $G$ is a
split extension then since $H^2(W; \Zt)=0$, any full subgroup is also a split 
extension, and the subgroups are in bijective correspondence to 
subgroups $S$ of $T^2$.

Now consider the space $\sub (T^2)$ of subgroups. By Pontrjagin
duality a subgroup $S\subseteq T^2$ corresponds to a subgroup $\Lambda^S:=\ker
((T^2)^*\lra S^*)$ of $\Z^2=(T^2)^*$. Thus the subgroups are
parametrized as follows. 
\begin{itemize}
\item There is one 2-dimensional group. 
\item The 1-dimensional groups $S$
correspond to 1-dimensional lattices $\Lambda^S$. 
We may think of the 1-dimensional lattices as lying over 
$\PP^1(\Q)$ with $S$ mapping to $\Lambda^S\tensor \Q$. The fibres of 
this map can be identified with the positive integers, as discussed at 
the end of \cite[Remark 6.7]{t2wqalg}.
\item The 0-dimensional groups $F$ correspond to 2-dimensional lattices 
$\Lambda^F$. 
\end{itemize}

\subsection{The space of subgroups (non-split extension)}
If $G$ is a non-split extension then one of the three elements of 
$T^2$ of order 2 is distinguished, because it is a square of an 
element outside $T^2$. Writing $x_e$ for the distinguished element, 
the full subgroups  are in bijective correspondence to 
subgroups of $T^2$ containing $x_0$. The corresponding subgroups of 
$\Lambda^S$ are precisely those lying in the index 2 sublattice 
$\Lambda^e\subseteq \Lambda^0$ corresponding to $x_e$. 
Because $W$ acts as negation, there is no further restriction on
$\Lambda^S$.  

The space of subgroups can be described exactly as in the split case
but with $\Lambda^0$ replaced by $\Lambda^e$. Accordingly, 
the component structure model is essentially the same for the split
and non-split cases. We will stick to the split case for
definiteness. 

\subsection{Component structure}
The Weyl group of a subgroup $S$ is the 
dual of $\Lambda^S/2\Lambda^S$, which is $C_2\times C_2$ for 
finite subgroups, $C_2$ for 1-dimensional subgroups and is trivial for 
$G$. This also gives the component structure since if $F\subset S$ we 
find $\Lambda^F\supset \Lambda^S$, and there is a map 
$(\Lambda^F/2\Lambda^F)\lla (\Lambda^S/2\Lambda^S)$ whose dual is the 
map in the component structure.

\subsection{The abelian models}
The models are controlled by  the punctured cubical diagram of rings: 
$$
\Rcospan=
\left(\begin{gathered}
\xymatrix{
  &\prod_H\Q \rrto \ddto &&\cIi_G\prod_H\Q \ddto\\
   && \Q \ddto \urto&\\
  &\prod_H\cIi_H\prod_{F} \Q\rrto &&\cIi_G \prod_H\cIi_H\prod_{F}\Q\\
\prod_F\Q\rrto \urto&&\cIi_G \prod_F\Q\urto&
}\end{gathered}\right) 
$$
The finite subgroups $F$ are the isolated points, the 1-dimensional
subgroups $H$ are the points isolated when the subgroups $F$ are removed.

\begin{defn}
(a) The {\em standard model} $\cA (G|\full)$ consists of cocartesian $R_{\bullet}$-modules, 
 which  consists of diagrams of the following form
$$\xymatrix{
  &N_1 \rrto \ddto &&N_{21}\ddto\\
 && N_2 \ddto \urto&\\
  &N_{10}\rrto &&N_{210}\\
N_0 \rrto \urto&&N_{20}\urto&
}$$
where $N_A$ is an $R_A$-module.

 (b) The {\em complete model} $\cAhat (G|\full)$ consists of equivariant 
sheaves over $\fX_G$. The data can be displayed on a diagram showing
the global sections.

$$\xymatrix{
  &\prod_HV(H)\rrto \ddto &&\cIi_G\prod_HV(H)\ddto\\
  \cF (\fX)\rrto \urto \ddto && V(G)\ddto \urto&\\
  &\prod_H\cIi_H\prod_{F} V(F)\rrto &&\cIi_G \prod_H\cIi_H\prod_{F}V(F)\\
\prod_FV(F)\rrto \urto&&\cIi_G \prod_FV(F)\urto&
}$$
\end{defn}

\begin{remark}
  We repeat the warning that the complete model is not the same as the model
  given by diagrams of sheaves of modules over the diagram
  $\mcR_{\bullet}$ of sheaves. This is most clearly seen at the $10$
  vertex. 

For example the module at $A=\{0\}$
  consists of $\prod_F V(F)$, but then (inserting more brackets for emphasis) leads to
  $$R_{10}\tensor_{R_0}\prod_F V(F)
  =\left[\prod_H\cIi_H\prod_F\Q\right]\tensor_{\prod_F\Q}\left[ \prod_F V(F)\right]$$
This is closely related to, but different from 
$$\prod_H\left[ \cIi_H\prod_{F} V(F)\right]. $$
\end{remark}

\end{document}